\documentclass[11pt,twoside]{amsart}

\usepackage{latexsym}
\usepackage{amssymb}
 \usepackage{vmargin}
\usepackage{amscd}
\usepackage{stmaryrd}
\usepackage{euscript}
\usepackage{mathrsfs}
\usepackage[all]{xy}

 \setmargins{32mm}{20mm}{14.6cm}{22cm}{1cm}{1cm}{1cm}{1cm}

\newcommand\la{\leftarrow}

\newcommand\id{\mathrm{id}}

\newcommand\ten{\otimes}
\newcommand\hten{\hat{\otimes}}
\newcommand\vareps{\varepsilon}

\renewcommand\H{\mathrm{H}}

\newcommand\N{\mathbb{N}}
\newcommand\Z{\mathbb{Z}}
\newcommand\Q{\mathbb{Q}}

\newcommand\bF{\mathbb{F}}

\newcommand\bZ{\mathbb{Z}}

\newcommand\C{\mathcal{C}}

\newcommand\sA{\mathscr{A}}

\newcommand\sF{\mathscr{F}}

\newcommand\sN{\mathscr{N}}
\newcommand\sO{\mathscr{O}}

\newcommand\fU{\mathfrak{U}}

\newcommand\fX{\mathfrak{X}}

\newcommand\m{\mathfrak{m}}

\newcommand\Mod{\mathrm{Mod}}

\newcommand\Hom{\mathrm{Hom}}

\newcommand{\llb}{\llbracket}
\newcommand{\rrb}{\rrbracket}

\newcommand\Pfd{\mathrm{Pfd}}

\newcommand\im{\mathrm{Im\,}}

\newcommand\Ch{\mathrm{Ch}}

\newcommand\Ab{\mathrm{Ab}}

\newcommand\Spec{\mathrm{Spec}\,}

\newcommand\Set{\mathrm{Set}}

\newcommand\Lim{\varprojlim}
\newcommand\LLim{\varinjlim}

\newcommand\into{\hookrightarrow}

\newcommand\xra{\xrightarrow}

\newcommand\bt{\bullet}
\newcommand\by{\times}

\newcommand\Tot{\mathrm{Tot}\,}

\newcommand\Strat{\mathrm{Strat}}
\newcommand\strat{\mathrm{strat}}
\newcommand\qDR{\mathrm{qDR}}

\newcommand\pd{\partial}

\newcommand\Zar{\mathrm{Zar}}

\newcommand\co{\colon\thinspace}

\newcommand\oR{\mathbf{R}}

\newcommand\oL{\mathbf{L}}

\newcommand\uleft\underleftarrow
\newcommand\uline\underline
\newcommand\uright\underrightarrow

\newtheorem{theorem}{Theorem}[section]
\newtheorem{proposition}[theorem]{Proposition}
\newtheorem{corollary}[theorem]{Corollary}

\newtheorem{lemma}[theorem]{Lemma}
\newtheorem*{theorem*}{Theorem}
\newtheorem*{proposition*}{Proposition}
\newtheorem*{corollary*}{Corollary}
\newtheorem*{lemma*}{Lemma}
\newtheorem*{conjecture*}{Conjecture}

\theoremstyle{definition}
\newtheorem{definition}[theorem]{Definition}

\newtheorem*{definition*}{Definition}

\theoremstyle{remark}

\newtheorem{remark}[theorem]{Remark}

\newtheorem*{example*}{Example}
\newtheorem*{examples*}{Examples}
\newtheorem*{remark*}{Remark}
\newtheorem*{remarks*}{Remarks}
\newtheorem*{exercise*}{Exercise}
\newtheorem*{property*}{Property}
\newtheorem*{properties*}{Properties}

\title{On $q$-de Rham cohomology via $\Lambda$-rings}

\author{J.P.Pridham}
\begin{document}

\begin{abstract}
We show that Aomoto's $q$-deformation of de Rham cohomology 
arises as a natural cohomology theory for $\Lambda$-rings. Moreover,  Scholze's $(q-1)$-adic completion of $q$-de Rham cohomology depends only on the Adams operations at each residue characteristic. This gives a fully functorial cohomology theory, including a lift of the Cartier isomorphism, for smooth formal schemes in mixed characteristic equipped with a suitable lift of Frobenius. If we attach $p$-power roots of $q$, the resulting theory is independent even of these lifts of Frobenius, refining a comparison by  Bhatt, Morrow and Scholze.
\end{abstract}


\maketitle

\section*{Introduction}

The $q$-de Rham cohomology of a polynomial ring is a $\Z[q]$-linear complex given by replacing the usual derivative with the Jackson $q$-derivative $\nabla_q(x^n)= [n]_qx^{n-1}dx$, where $[n]_q$ is Gauss' $q$-analogue $\frac{q^n-1}{q-1}$ of the integer $n$. In \cite{scholzeqdef}, Scholze discussed the $(q-1)$-adic completion of this theory for smooth rings, explaining relations to $p$-adic Hodge theory and singular cohomology, and conjecturing that it is   independent of co-ordinates, so functorial for smooth algebras over a fixed base \cite[Conjectures 1.1, 3.1 and 7.1]{scholzeqdef}. 

We show that $q$-de Rham cohomology with $q$-connections naturally arises as a functorial invariant of $\Lambda$-rings (Theorems \ref{mainthm},  \ref{mainthm2} and Proposition \ref{qconnprop}), and that its 
$(q-1)$-adic completion depends only on a $\Lambda_P$-ring structure (Theorem \ref{mainthm3}), for $P$ the set of residue characteristics; a $\Lambda_P$-ring has a lift of Frobenius for each $p \in P$. This recovers the known equivalence between de Rham cohomology and complete $q$-de Rham cohomology over the rationals, while giving no really new functoriality statements for smooth schemes over $\Z$. However, in mixed characteristic, it means that complete $q$-de Rham cohomology depends only on a lift $\Psi^p$ of absolute Frobenius locally generated by co-ordinates with $\Psi^p(x_i)=x_i^p$. Given such data, we construct (Proposition  \ref{cartierprop}) a  quasi-isomorphism between  Hodge cohomology and $q$-de Rham cohomology modulo $[p]_q$, extending the local lift of the  Cartier  isomorphism in \cite[Proposition 3.4]{scholzeqdef}.

Taking the Frobenius stabilisation of the complete $q$-de Rham complex of $A$  yields a complex resembling   the de Rham--Witt complex. We show (Theorem \ref{mainthm4}) that up to $(q^{1/p^{\infty}}-1)$-torsion, the $p$-adic completion of this complex depends only on the $p$-adic completion of $A[\zeta_{p^{\infty}}]$ (where $\zeta_n$ denotes a primitive $n$th root of unity), with no requirement for a lift of Frobenius or a choice of co-ordinates. 
The main idea is to show that the stabilised $q$-de Rham complex is in a sense given by applying Fontaine's period ring construction $A_{\inf}$ to  the best possible perfectoid approximation to $A[\zeta_{p^{\infty}}]$. 
As a consequence, this shows (Corollary \ref{finalcor}) that after attaching all $p$-power roots of $q$, $q$-de Rham cohomology in mixed characteristic is independent of choices, which was already known after base change to a period ring, via the comparisons of  \cite{BhattMorrowScholze} between $q$-de Rham cohomology and their theory $A\Omega$. 
 
The  cohomology theories we construct thus depend either on Adams operations at the residue characteristics (for de Rham) or on $p$-power roots of $q$ (for variants of de Rham--Witt), establishing correspondingly weakened versions of the conjectures of \cite{scholzeqdef}; in Remark \ref{finalrmk}, we suggest a possible candidate for a theory without those restrictions. 
The essence of our construction of $q$-de Rham cohomology of $A$ over $R$  is to set $q$ to be an element of rank $1$ for the $\Lambda$-ring structure, and to look at flat $\Lambda$-rings $B$ over $R[q]$ equipped with morphisms $A \to B/(q-1)$ of $\Lambda$-rings over $R$. If these seem unfamiliar, reassurance should be provided by the observation that $(q-1)B$  carries $q$-analogues of divided power operations (Remark \ref{DPrmk}).
For the variants of de Rham--Witt cohomology in \S \ref{DRWsn}, the key to giving a characterisation independent of lifts of Frobenius is the factorisation of the tilting equivalence for perfectoid algebras via a category of $\Lambda_p$-rings, leading to constructions similar to \cite{BhattMorrowScholze}.  

I would like to thank Peter Scholze for many helpful comments,  in particular  about the possibility of a  $q$-analogue of de Rham--Witt cohomology, and Michel Gros for spotting a missing hypothesis. I would also like to thank the anonymous referee for suggesting many improvements.

\tableofcontents

\section{Comparisons for $\Lambda$-rings}

We will follow standard notational conventions for $\Lambda$-rings. These are commutative rings equipped with operations $\lambda^i$ resembling alternating powers, in particular satisfying $\lambda^k(a+b)= \sum_{i=0}^{k} \lambda^i(a)\lambda^{k-i}(b)$, with $\lambda^0(a)= 1$ and $\lambda^1(a)=a$. For background, see \cite{borgerBasicGeomI} and references therein. The $\Lambda$-rings we encounter are all torsion-free, in which case \cite{wilkerson} shows the $\Lambda$-ring structure is equivalent to giving ring endomorphisms $\Psi^n$ for $n \in \Z_{>0}$ with $\Psi^{mn}=\Psi^m \circ \Psi^n$ and $\Psi^p(x) \equiv x^p \mod p$ for all primes $p$. 
If we write  $\lambda_t(f):=\sum_{i \ge 0} \lambda^i(f)t^i$ and $\Psi_t(f):= \sum_{n \ge 1} \Psi^n(f)t^n$, then the families of operations are related by the formula $\Psi_t= -t\frac{d\log\lambda_{-t} }{d t}$.

We refer to elements $x$   with $\lambda^i(x)=0$ for all $i>1$ (or equivalently $\Psi^n(x)=x^n$ for all $n$) as elements of rank $1$.

\subsection{The $\Lambda$-ring $\Z[q]$}

%
\begin{definition}
 Define $\Z[q]$ to be the $\Lambda$-ring with operations determined by setting $q$ to be of rank $1$.
\end{definition}

We now consider the $q$-analogues $[n]_q:= \frac{q^n-1}{q-1} \in \Z[q]$ of the integers, with $[n]_q!=[n]_q[n-1]_q\ldots  [1]_q$, and $\binom{n}{k}_q= \frac{[n]_q!}{[n-k]_q![k]_q!}$.

\begin{remark}
 To see the importance of regarding $\Z[q]$ as a $\Lambda$-ring observe that the binomial expressions
\begin{align*}
 \lambda^k(n)=  \tbinom{n}{k},\quad
\lambda^k(-n)= (-1)^k\tbinom{n+k-1}{k}
\end{align*}
have as $q$-analogues the Gaussian binomial theorems
\begin{align*}
 \lambda^k([n]_q)=  q^{k(k-1)/2}\tbinom{n}{k}_q,\quad
\lambda^k(-[n]_q) = (-1)^k\tbinom{n+k-1}{k}_q,
\end{align*}
as well as Adams operations
\[
 \Psi^i([n]_q)= [n]_{q^i}.
\]
\end{remark}

For any torsion-free $\Lambda$-ring, localisation at a set of elements closed under the Adams operations always yields another  $\Lambda$-ring, since  $\Psi^p(a^{-1})-a^{-p} = (\Psi^p(a)a^p)^{-1}(a^p- \Psi^p(a))$ is divisible by $p$.

\begin{lemma}\label{lambdalemma}
 For the $\Lambda$-ring structure on $\Z[x,y]$ with  $x,y$ of rank $1$, the  elements  
\[
 \lambda^n(\tfrac{y-x}{q-1}) \in\Z[q, \{(q^n-1)^{-1}\}_{n \ge 1},x,y]
\]
are given by
\begin{align*}
 \lambda^k(\tfrac{y-x}{q-1}) &=\frac{(y-x)(y-qx)\ldots(y-q^{k-1}x)}{(q-1)^k[k]_q!},\\
&= \sum_{j=0}^k\frac{q^{j(j-1)/2} (-x)^jy^{k-j}}{[j]_q![k-j]_q!}.
\end{align*}
\end{lemma}
\begin{proof}
 The second expression comes from multiplying out the Gaussian binomial expansions. The easiest way to prove the first is to observe that  $\lambda^k(\tfrac{y-x}{q-1})$ must be a homogeneous polynomial of degree $k$ in $x,y$, with coefficients in the integral domain $\Z[q, \{(q^n-1)^{-1}\}_{n \ge 1}]$, and  to note that
\[
 \lambda^k(\tfrac{q^nx-x}{q-1})= \lambda^k([n]_qx)= q^{k(k-1)/2}\tbinom{n}{k}_qx^k.
\]
Thus $\lambda^k(\frac{y-x}{q-1})$  agrees with the homogeneous polynomial above for infinitely many values of $\frac{y}{x}$, so must be equal to it.
\end{proof}

\begin{remark}\label{DPrmk}
Note that as $q\to 1$, Lemma \ref{lambdalemma} gives $ (q-1)^k\lambda^k(\frac{y-x}{q-1})\to \frac{(x-y)^k}{k!}$. Indeed, for any rank $1$ element $x$ in a $\Lambda$-ring we have  
\[
 \lambda_{(q-1)t} (\frac{x}{q-1})= \sum_{k \ge 0} \frac{(xt)^k}{[k]_q!},
\]
which is just the $q$-exponential $e_q(xt)$. Multiplicativity and universality then imply   that $\lambda_{(q-1)t} (\frac{a}{q-1})$ is a $q$-deformation of $\exp(at)$ for all $a$. Thus $(q-1)^k\lambda^k(\frac{a}{q-1})$ is  a $q$-analogue of the $k$th divided power $(a^k/k!)$. 
An explicit expression  comes recursively from the formula
\[
[k]_q(q-1)\lambda^k(\tfrac{a}{q-1})=\sum_{i>0} \lambda^i(a)\lambda^{k-i}(\tfrac{a}{q-1}),
\]
obtained by subtracting $\lambda_t(\frac{a}{q-1})$ from each side of the expression $ \lambda_{qt}(\frac{a}{q-1})=\lambda_t(a)\lambda_t(\frac{a}{q-1})$, which arises because $q$ is of rank $1$ and $\frac{qa}{q-1}= a+\frac{a}{q-1}$.
\end{remark}

\begin{lemma}\label{subringbasis}
 For elements $x,y$ of rank $1$, the $\Lambda$-subring of $\Z[q, \{(q^n-1)^{-1}\}_{n \ge 1},x,y]$ generated by $q,x,y, \frac{y-x}{q-1}$ has basis      $\lambda^k(\frac{y-x}{q-1})$ as a $\Z[q,x]$-module.
\end{lemma}
\begin{proof}
The $\Lambda$-subring clearly contains the $\Z[q,x]$-module $M$ generated by the elements $\lambda^k(\frac{y-x}{q-1})$, which are also clearly $\Z[q,x]$-linearly independent. Since $\Z[q,x]$ is a $\Lambda$-ring, it suffices to show that $M$ is closed under multiplication. 

By Lemma \ref{lambdalemma}, we know that 
\[
 \lambda^i(\tfrac{y-x}{q-1})\lambda^j(\tfrac{y-q^ix}{q-1})= \tbinom{i+j}{i}_q\ \lambda^{i+j}(\tfrac{y-x}{q-1}).
\]
We can rewrite $\frac{y-q^ix}{q-1}= \frac{y-x}{q-1}- [i]_qx$, so $\lambda^j(\frac{y-q^ix}{q-1})-\lambda^j(\frac{y-x}{q-1})$ lies in the $\Z[q,x]$-module spanned by $\lambda^m(\frac{y-x}{q-1}) $ for $m<j$. By induction on $j$, it thus follows that
\[
 \lambda^i(\tfrac{y-x}{q-1})\lambda^j(\tfrac{y-q^ix}{q-1}) - \lambda^i(\tfrac{y-x}{q-1})\lambda^j(\tfrac{y-x}{q-1}) \in M,
\]
so the binomial expression above implies $ \lambda^i(\frac{y-x}{q-1})\lambda^j(\frac{y-x}{q-1}) \in M $.
\end{proof}

\subsection{$q$-cohomology of $\Lambda$-rings }

\begin{definition}
 Given a $\Lambda$-ring $R$, say that $A$ is a $\Lambda$-ring over $R$ if it is a $\Lambda$-ring equipped with a morphism $R \to A$ of $\Lambda$-rings. We say that $A$ is a flat $\Lambda$-ring over $R$ if $A$ is flat as a module over the commutative ring   underlying $R$.
\end{definition}

\begin{definition}\label{Stratqdef}
Given a morphism $R \to A$ of $\Lambda$-rings, we define the category $\Strat^q_{A/R}$ to consist of flat $\Lambda$-rings $B$ over $R[q]$ equipped with a compatible  morphism $f \co A \to B/(q-1)$, such that 
 $f$ admits a lift to $B$; a choice of lift is not taken to be part of the data, so need not be preserved by morphisms.

More concisely,  $\Strat^q_{A/R}$ is the Grothendieck construction of the set-valued functor
\[
(\Spec A)_{\strat}^q\co B \mapsto \im(\Hom_{\Lambda,R}(A,B)\to \Hom_{\Lambda,R}(A,B/(q-1))) 
\]
on the category $f\Lambda(R[q])$ of flat $\Lambda$-rings over $R[q]$. 
\end{definition}

\begin{definition}\label{qDRdef}
 Given a flat morphism $R \to A$ of $\Lambda$-rings, define $\qDR(A/R)$ to be the cochain complex of $R[q]$-modules given by taking the homotopy limit (in the sense of \cite{bousfieldkan}) of
 the functor
\begin{align*}
 \Strat^q_{A/R} &\to \Ch(R[q])\\
B &\mapsto B.
\end{align*}
The cochain complex $\qDR(A/R)$ naturally carries $(R[q], \Psi^n)$-semilinear operations $\Psi^n$ coming from the morphisms $\Psi^n \co B\ten_{R[q], \Psi^n}R[q] \to B$ of $R[q]$-modules, for $B \in \Strat^q_{A/R}$.
\end{definition}
 Equivalently, can we follow the approach of 
\cite{Gr,simpsonHtpy} towards the stratified site and de Rham stack by regarding $\qDR(A/R)$
as the quasi-coherent cohomology complex of $(\Spec A)_{\strat}^q$, as follows. 
\begin{definition}\label{RHomdef}
 Given a category $\C$, write
$[\C,\Set]$ and $[\C,\Ab]$ for the categories of functors on $\C$ taking values in  sets and abelian groups, respectively. For any functor $X \co \C \to \Set$,  we then denote by $\oR\Hom_{[\C,\Set]}(X,-)$ the functor from $[\C,\Ab]$ to cochain complexes given by taking the right-derived functor of the functor
\[
\Hom_{[\C,\Set]}(X,-)\co [\C,\Ab] \to \Ab
\] 
of natural transformations with source $X$.
\end{definition}

For the forgetful functor $\sO\co f\Lambda(R[q])\to \Mod(R[q])$   to the category of $R[q]$-modules,
we then have
\[
 \qDR(A/R)= \oR\Hom_{[f\Lambda(R[q]),\Set]}((\Spec A)_{\strat}^q,\sO),
\]
with Adams operations $\Psi^n \co \sO\ten_{R[q], \Psi^n}R[q] \to \sO$ giving the $(R[q], \Psi^n)$-semilinear operations $\Psi^n$   on $\qDR(A/R)$.

\begin{remark}\label{cosimplicialrmk}
 The cochain complex $\qDR(A/R)$ naturally carries much more structure than these Adams operations. Whenever we can factor the functor $\sO$ through a model category $\C$  equipped with a forgetful functor to  $\Ch(R[q])$ preserving weak equivalences and homotopy limits, we can regard  $\qDR(A/R)$ as an object of the homotopy category of $\C$ by taking the defining homotopy limit in $\C$. 

The universal such example for $\C$ is given by the model category of cosimplicial $\Lambda$-rings over $R[q]$, with weak equivalences being quasi-isomorphisms (i.e. cohomology isomorphisms) and fibrations being surjections; the underlying  cochain complex has differential $\sum (-1)^j \pd^j$. That this determines a model structure follows from Kan's transfer theorem  \cite[Theorem 11.3.2]{Hirschhorn} applied to the cosimplicial Dold--Kan normalisation functor taking values in unbounded chain complexes with the projective model structure; the conditions of that theorem are satisfied because the left adjoint functor sends acyclic cofibrant complexes to cosimplicial $\Lambda$-rings which automatically have a contracting homotopy in the form of an extra codegeneracy map.

In particular, $\qDR(A/R) $ naturally underlies a quasi-isomorphism class of cosimplicial $\Lambda$-rings over $R[q]$; forgetting the $\lambda$-operations gives a cosimplicial commutative $R[q]$-algebra, and stabilisation then gives an $E_{\infty}$-algebra over $R[q]$, all with underlying cochain complex  $\qDR(A/R) $. 
\end{remark}

\begin{definition}
 Given a polynomial ring $R[x]$, recall from \cite{scholzeqdef} that the $q$-de Rham (or Aomoto--Jackson) cohomology $q\mbox{-}\Omega^{\bt}_{ R[x]/R}$ is given by the complex
\[
 R[x][q] \xra{\nabla_q}  R[x][q]dx,\quad\text{ where }\quad \nabla_q(f) =  \frac{f(qx) -f(x)}{x(q-1)}dx,
\]
so $\nabla_q(x^n)= [n]_qx^{n-1}dx$. 

Given a polynomial ring $R[x_1, \ldots,x_d]$, the $q$-de Rham complex $q\mbox{-}\Omega^{\bt}_{ R[x_1, \ldots,x_d]/R}$ is then set to be
\[
 q\mbox{-}\Omega^{\bt}_{ R[x_1]/R}\ten_{R[q]} q\mbox{-}\Omega^{\bt}_{ R[x_2]/R}\ten_{R[q]}\ldots \ten_{R[q]} q\mbox{-}\Omega^{\bt}_{ R[x_d]/R},
\]
so takes the form
\[
 R[x_1, \ldots,x_d][q] \xra{\nabla_q} \Omega^1_{R[x_1, \ldots,x_d]/R}[q] \xra{\nabla_q} \ldots \xra{\nabla_q}\Omega^d_{R[x_1, \ldots,x_d]/R}[q].
\]
\end{definition}

\begin{definition}\label{tildeXqdef}
 Given a flat morphism $R \to A$ of $\Lambda$-rings with $X=\Spec A$,  define the 
functor $\tilde{X}_{\strat}^q$ from  flat $\Lambda$-rings over $R[q]$ to simplicial sets
 by taking the \v Cech nerve of $\Hom_{\Lambda,R}(A,B)\to \Hom_{\Lambda,R}(A,B/(q-1))$, so
\begin{align*}
 (\tilde{X}_{\strat}^q)_n(B)&:=  \overbrace{\Hom_{\Lambda,R}(A,B)\by_{\Hom_{\Lambda,R}(A,B/(q-1))} \ldots \by_{\Hom_{\Lambda,R}(A,B/(q-1))} \Hom_{\Lambda,R}(A,B) }^{n+1}\\
&= \Hom_{\Lambda,R}(A, \overbrace{B\by_{B/(q-1)}\ldots \by_{B/(q-1)} B}^{n+1}),
\end{align*}
with simplicial operations
\begin{align*}
 \pd_j(f_0,f_1, \ldots, f_n)&:= (f_0,f_1, \ldots, f_{j-1}, f_{j+1}, f_{j+2}, \ldots, f_n), \\
\sigma_j(f_0,f_1, \ldots, f_n) & := (f_0,f_1, \ldots, f_{j},f_j, f_{j+1}, \ldots, f_n). 
\end{align*}
\end{definition}

\begin{definition}\label{DKdef}
Given a cosimplicial abelian group $V^{\bt}$, we write $NV$ for the Dold--Kan normalisation of $V$ (\cite[Lemma 8.3.7]{W} applied the opposite category). This is a cochain complex with $N^rV= V^r \cap_{j <r} \ker \sigma^j$ and differential $d= \sum_{j=0}^{r+1} (-1)^j\pd^j \co N^rV \to N^{r+1}V$. 
\end{definition}

\begin{lemma}\label{Qcalclemma}
If, for $X=\Spec A$, the functors $(\tilde{X}_{\strat}^q)_n$ are represented by flat $\Lambda$-rings $\Gamma( (\tilde{X}_{\strat}^q)_n,\sO)$ over $R[q]$, then  a  model for $\qDR(A/R)$ is given by the Dold--Kan normalisation  of the cosimplicial module $n \mapsto \Gamma( (\tilde{X}_{\strat}^q)_n,\sO)$.
\end{lemma}
\begin{proof}
 The set-valued functor $X_{\strat}^q=(\Spec A)_{\strat}^q$ of Definition \ref{Stratqdef}  is resolved by the simplicial functor $\tilde{X}_{\strat}^q$ of Definition \ref{tildeXqdef}. In the notation of Definition \ref{RHomdef}, this implies that the functor $\Hom_{[f\Lambda(R[q]),\Set]}(X_{\strat}^q,-)$ on $[f\Lambda(R[q]),\Ab]$ is resolved by the cochain complex
\[
N \Hom_{[f\Lambda(R[q]),\Set]}((\tilde{X}_{\strat}^q)_{\bt},-).
\]

Although $X_{\strat}^q$  is not representable on the category of flat $\Lambda$-rings over $R[q]$,  our hypotheses ensure that each functor $(\tilde{X}_{\strat}^q)_n$ is so. Thus the functors  $\Hom_{[f\Lambda(R[q]),\Set]}((\tilde{X}_{\strat}^q)_n,-)$ and their direct summands $N^n \Hom_{[f\Lambda(R[q]),\Set]}((\tilde{X}_{\strat}^q)_{\bt},-)$ are exact, and are their own right-derived functors. This implies that the cochain complex of functors above models $\oR\Hom_{[f\Lambda(R[q]),\Set]}(X_{\strat}^q,-)$, and the result follows by evaluation at $\sO$.
\end{proof}

\begin{proposition}\label{Qcalc}
 If $R$ is a $\Lambda$-ring and $x$ of rank $1$, then $\qDR(R[x]/R)$ can be calculated by Dold--Kan normalisation of the cosimplicial $R[q]$-module $U^{\bt}$ given by setting $U^n$ to be the $\Lambda$-subring   
\[
U^n \subset R[q, \{(q^m-1)^{-1}\}_{m \ge 1},x_0, \ldots, x_n]
\] 
generated by  $q$ and the elements $x_i$ and  $\frac{x_i-x_j}{q-1}$, with cosimplicial operations 
\[
 \pd^jx_i := \begin{cases}
              x_i & j>i \\ x_{i+1} & j \le i,  
             \end{cases}
\quad \quad
\sigma^jx_i := \begin{cases}
                  x_i & j\ge i \\ x_{i-1} & j < i.
                \end{cases}  
\]
\end{proposition}
\begin{proof}
We verify the conditions of Lemma \ref{Qcalclemma} by showing that each $U^n$ is a flat $\Lambda$-ring over $R[q]$ representing $(\tilde{X}_{\strat}^q)_n$.
Taking  $X=\Spec R[x]$, observe that any element of $(\tilde{X}_{\strat}^q)_n(B)$ gives rise to a morphism $f \co R[q,x_0, \ldots, x_n] \to B$ of  $\Lambda$-rings over $R[q]$, with the image of $x_i-x_j$ divisible by $(q-1)$. Flatness of $B$ then gives a unique element $f(x_i-x_j)/(q-1) \in B$, so we have a map $f$ to $B$  from the free  $\Lambda$-ring $L$ over $R[q,x_0, \ldots, x_n]$ generated by  elements $z_{ij}$ with $(q-1)z_{ij} =x_i-x_j$. 

Since  $B$ is flat, it embeds in $B[\{(q^m-1)^{-1}\}_{m \ge 1} ]$ (the only hypothesis we really need) implying that the image of $f$ factors through  the image $U^n$ of $L$ in $R[q, \{(q^m-1)^{-1}\}_{m \ge 1},x_0, \ldots, x_n]$. To see that $(\tilde{X}_{\strat}^q)_n$ is represented by $U^n$,  we only now need to check that $U^n$ is itself flat over $R[q]$, which follows because the argument of Lemma \ref{subringbasis} gives a  basis
\[
 x_0^{r_0}\lambda^{r_1}(\tfrac{x_1-x_0}{q-1})\cdots \lambda^{r_n}(\tfrac{x_n-x_{n-1}}{q-1})
\]
for $U^n$ over $R[q]$. We therefore have $\qDR( R[x]/R)\simeq NU^{\bt}$.
\end{proof}

In fact, the proofs of Lemma \ref{Qcalclemma} and Proposition \ref{Qcalc} show that the natural  cosimplicial $\Lambda$-ring structure on  $U^{\bt}$ gives a model for the  cosimplicial $\Lambda$-ring structure on $\qDR(R[x]/R)$ coming from Remark \ref{cosimplicialrmk}.

\begin{definition}
 Following \cite[Proposition 5.4]{scholzeqdef}, we denote by $\oL \eta_{(q-1)}$ the d\'ecalage functor with respect to the  derived $(q-1)$-adic filtration. This is given on complexes $C^{\bt}$ of $(q-1)$-torsion-free $R[q]$-modules by 
\[
 (\eta_{(q-1)}C)^n := \{ c \in (q-1)^nC^n ~:~ dc \in (q-1)^{n+1}C^{n+1}\},
\]
and is extended to the derived category of $R[q]$-modules by taking torsion-free resolutions.
\end{definition}

\begin{theorem}\label{mainthm} 
 If $R$ is a $\Lambda$-ring and if the polynomial ring  $R[x_1, \ldots,x_n]$ is given the $\Lambda$-ring structure for which the elements $x_i$ are of rank $1$, then there are $R[q]$-linear zigzags of quasi-isomorphisms
\begin{align*}
 \qDR( R[x_1, \ldots,x_n]/R) &\simeq (\Omega^*_{R[x_1, \ldots,x_n]/R}[q], (q-1)\nabla_q)\\
 \oL\eta_{(q-1)} \qDR( R[x_1, \ldots,x_n]/R) &\simeq q\mbox{-}\Omega^{\bt}_{R[x_1, \ldots,x_n]/R}.
\end{align*}
\end{theorem}
\begin{proof}
It suffices to prove the first statement, the second following immediately by d\'ecalage. We have $(\Spec A\ten_RA')_{\strat}^q(B)= (\Spec A)_{\strat}^q(B)\by (\Spec A')_{\strat}^q(B)$, and similarly for the simplicial functor $\widetilde{(\Spec A\ten_RA')}_{\strat}^q$ of Definition \ref{tildeXqdef}. Since
coproduct of flat $\Lambda$-rings over $R[q]$ is given by $\ten_{R[q]}$, it follows from Lemma \ref{Qcalclemma} and Proposition \ref{Qcalc} that  $\qDR(R[x_1, \ldots,x_n]/R)$ can be calculated as the Dold--Kan normalisation of   $(U^{\bt})^{\ten_{R[q]}n}$ (given by the $n$-fold tensor product $ (U^m)^{\ten_{R[q]}n}$ in cosimplicial level $m$), for the cosimplicial module $U^{\bt}$ of Proposition \ref{Qcalc}.   

The proof now proceeds in a similar fashion to the comparison between crystalline and de Rham cohomology in \cite{BhattDeJong}. We consider the cochain complexes $\tilde{\Omega}^{\bt}(U^m)$ given by
\[
 U^m \xra{(q-1)\nabla_q} \bigoplus_i  U^mdx_i \xra{(q-1)\nabla_q}\bigoplus_{i<j}  U^mdx_i\wedge dx_j  \xra{(q-1)\nabla_q}\ldots .
\]
In order to see that this differential takes values in the codomains given,  observe that 
\begin{align*}
  (q-1)\nabla_{q,y}\lambda^k(\tfrac{y-x}{q-1})&= y^{-1}(\lambda^k(\tfrac{qy-x}{q-1}) - \lambda^k(\tfrac{y-x}{q-1}))dy\\
&= y^{-1}(\lambda^k(y+ \tfrac{y-x}{q-1}) - \lambda^k(\tfrac{y-x}{q-1}))dy\\
&= \lambda^{k-1}(\tfrac{y-x}{q-1})dy,
\end{align*}
and similarly
\[
 (q-1)\nabla_{q,x}\lambda^k(\tfrac{y-x}{q-1})
=  \sum_{i \ge 1} (-1)^ix^{i-1}\lambda^{k-i}(\tfrac{y-x}{q-1})dx. 
\]

The first calculation   also shows that the inclusion  $\tilde{\Omega}^{\bt}(U^{m-1})\into \tilde{\Omega}^{\bt}(U^m)$ is a quasi-isomorphism, since for $\omega \in\tilde{\Omega}^{\bt}(U^{m-1})$, we have 
\[
(q-1)\nabla_{q,x_m}\omega\lambda^k(\tfrac{x_m-x_{m-1}}{q-1})= \omega\lambda^{k-1}(\tfrac{x_m-x_{m-1}}{q-1})dx_m
\]
for $k \ge 1$,
allowing us to define a contracting homotopy 
 \begin{align*}
 h(\omega\lambda^{k-1}(\tfrac{x_m-x_{m-1}}{q-1})dx_m) &:=  \omega\lambda^k(\tfrac{x_m-x_{m-1}}{q-1}),\\
h(\omega\lambda^{k-1}(\tfrac{x_m-x_{m-1}}{q-1})) &:=0.
 \end{align*}

Since contracting homotopies interact well with tensor products, it also follows that the inclusion  $\tilde{\Omega}^{\bt}(U^{m-1})^{\ten_{R[q]}n}\into \tilde{\Omega}^{\bt}(U^m)^{\ten_{R[q]}n}$ is a quasi-isomorphism. By induction on $m$ we deduce that the inclusions $\tilde{\Omega}^{\bt}(U^0)^{\ten_{R[q]}n}\into \tilde{\Omega}^{\bt}(U^m)^{\ten_{R[q]}n}$, and hence their retractions given by diagonals $U^m \to U^0$, are quasi-isomorphisms. These combine to give a  quasi-isomorphism
\[
   \Tot N(\tilde{\Omega}^{\bt}(U^{\bt})^{\ten_{R[q]}n}) \to \tilde{\Omega}^{\bt}(U^0)^{\ten_{R[q]}n}= \tilde{\Omega}^{\bt}(R[x])^{\ten_{R[q]}n}
\]
on total complexes of normalisations. 

Now, the cosimplicial module $\tilde{\Omega}^r(U^{\bt})$ is given by  the cosimplicial (i.e. levelwise) tensor product of $U^{\bt}$ with the cosimplicial $\Z$-module  
\[
j \mapsto \bigoplus_{0 \le i_1< i_2< \ldots < i_r \le j} \Z dx_{i_1}\wedge \ldots \wedge d x_{i_r},
\]
 with operations induced by those in Proposition \ref{Qcalc}.
For $r>0$, this cosimplicial $\Z$-module is contractible, via the extra codegeneracy map given by 
\[
\sigma^{-1}(dx_{i_1}\wedge \ldots \wedge dx_{i_r}) = \begin{cases}
                                                        dx_{i_1-1}\wedge \ldots \wedge dx_{i_r-1} & i_1>0, \\
0 & i_1=0.
\end{cases}
\]
The Eilenberg--Zilber theorem (\cite[\S 8.5]{W} applied to the opposite category) ensures that the normalisation of a cosimplicial tensor product is quasi-isomorphic to the tensor product of the normalisations.  
Tensoring with a complex which has an extra codegeneracy map always produces an acyclic complex,   so 
$\tilde{\Omega}^r(U^{\bt})$ and its tensor powers are all acyclic for $r>0$. 

The brutal truncation maps 
\[
\Tot N(\tilde{\Omega}^{\bt}(U^{\bt})^{\ten_{R[q]}n}) \to N(U^{\bt})^{\ten_{R[q]}n}\simeq \qDR( R[x_1, \ldots,x_n]/R) 
\]
are therefore quasi-isomorphisms
 of flat cochain complexes over $R[q]$, so 
\[
 \qDR( R[x_1, \ldots,x_n]/R) \simeq \tilde{\Omega}^{\bt}(R[x])^{\ten_{R[q]}n},
\]
and we just observe that    $\tilde{\Omega}^{\bt}(R[x]) =(\Omega^*_{R[x]/R}[q], (q-1)\nabla_q)$.
\end{proof}

\begin{remark}
 Note that Theorem \ref{mainthm} and Remark \ref{cosimplicialrmk} together imply that $q\mbox{-}\Omega^{\bt}_{R[x_1, \ldots,x_n]/R}$ naturally underlies the d\'ecalage of a cosimplicial $\Lambda$-ring over $R[q]$. Even the underlying cosimplicial commutative ring structure carries more information than an $E_{\infty}$-structure when $\Q \nsubseteq R$. 
\end{remark}

\subsection{Completed $q$-cohomology}

\begin{definition}
Given a morphism $R \to A$ of $\Lambda$-rings, we define the category $\hat{\Strat}^q_{A/R} \subset \Strat^q_{A/R}$ to consist of those objects which are $(q-1)$-adically complete.

Equivalently, $\hat{\Strat}^q_{A/R}$ is the Grothendieck construction of the functor
\[
\widehat{(\Spec A)}_{\strat}^q\co  B \mapsto \im(\Hom_{\Lambda,R}(A,B)\to \Hom_{\Lambda,R}(A,B/(q-1)))
\]
on the category of flat $(q-1)$-adically complete $\Lambda$-rings over $R[q]$. 
\end{definition}

\begin{definition}
 Given a flat morphism $R \to A$ of $\Lambda$-rings, define $\widehat{\qDR}(A/R)$ to be the cochain complex of $R\llbracket q-1 \rrbracket$-modules given by taking the homotopy limit of
 the functor
\begin{align*}
 \hat{\Strat}^q_{A/R} &\to \Ch(R\llbracket q-1 \rrbracket)\\
B &\mapsto B.
\end{align*}
\end{definition}

The following is immediate:
\begin{lemma}
  Given a flat morphism $R \to A$ of $\Lambda$-rings, the complex $\widehat{\qDR}(A/R)$ is the derived $(q-1)$-adic completion of $\qDR(A/R)$.
\end{lemma}


\begin{definition}\label{qhatdef2}
 As in \cite[\S 3]{scholzeqdef}, given a formally \'etale map $\boxempty\co R[x_1, \ldots,x_d] \to A$, define $\widehat{q\mbox{-}\Omega}^{\bt}_{A/R, \boxempty}$ to be the complex
\[
 A\llbracket q-1\rrbracket \xra{\nabla_q} \Omega^1_{A/R}\llbracket q-1\rrbracket\xra{\nabla_q} \ldots \xra{\nabla_q} \Omega^d_{A/R}\llbracket q-1\rrbracket,
\]
where $\nabla_q$ is defined as follows. First note that the $R\llbracket q-1\rrbracket$-linear  ring endomorphisms $\gamma_i$ of $R[x_1, \ldots,x_d]\llbracket q-1\rrbracket$ given by $\gamma_i(x_j)= q^{\delta_{ij}}x_j$ extend uniquely to endomorphisms of $A\llbracket q-1\rrbracket$ which are the identity modulo $(q-1)$, then set
\[
 \nabla_q(f):= \sum_i \frac{\gamma_i(f)-f}{(q-1)x_i}dx_i.
\]
\end{definition}

Note that $\widehat{q\mbox{-}\Omega}^{\bt}_{ R[x_1, \ldots,x_d]/R}$ is just the $(q-1)$-adic completion of $ q\mbox{-}\Omega^{\bt}_{ R[x_1, \ldots,x_d]/R}$.

\begin{theorem}\label{mainthm2}
 If $R$ is a flat $\Lambda$-ring over $\Z$ and $\boxempty\co R[x_1, \ldots,x_d]\to A$ is a formally \'etale map of $\Lambda$-rings,  the elements $x_i$ having rank $1$, then there are zigzags of $R\llbracket q\rrbracket$-linear quasi-isomorphisms
\begin{align*}
\widehat{\qDR}( A/R) &\simeq  (\Omega^*_{A/R}\llb q -1\rrb, (q-1)\nabla_q), &
 \oL\eta_{(q-1)} \widehat{\qDR}( A/R) &\simeq \widehat{q\mbox{-}\Omega}^{\bt}_{A/R, \boxempty}.
\end{align*}
 The induced quasi-isomorphisms
\begin{align*}
\widehat{\qDR}( A/R)\ten^{\oL}_{R\llb q-1\rrb}R  &\simeq  (\Omega^*_{A/R}, 0), &
 (\oL\eta_{(q-1)} \widehat{\qDR}( A/R))\ten^{\oL}_{R\llb q-1\rrb}R &\simeq \Omega^{\bt}_{A/R}
\end{align*}
are independent of the choice of framing.
\end{theorem}
\begin{proof}
Since the framing $\boxempty$ is formally \'etale,  for any $(q-1)$-adically complete commutative $R[q]$-algebra $B$, any commutative square
\[
 \xymatrix{ R[x_1, \ldots,x_d] \ar[r] \ar[d]_{\boxempty} & B \ar[d] \\
A \ar[r] \ar@{-->}[ur] & B/(q-1).
}
\]
of $R$-algebra homomorphisms admits a unique dashed arrow as shown.

For any $(q-1)$-adically complete flat $\Lambda$-ring $B$ over $R$, we then have the same property for $\Lambda$-ring homomorphisms over $R$ instead of $R$-algebra homomorphisms: 
the diagram above gives a unique dashed $R$-algebra homomorphism, and uniqueness of lifts ensures that it commutes with Adams operations, so is a $\Lambda$-ring homomorphism ($R$ being flat over $\Z$). 
Similarly (taking $B= A\llb q-1 \rrb$) uniqueness of lifts ensures that the operations $\gamma_i$  are  $\Lambda$-ring  endomorphisms of $A\llb q-1 \rrb$.

We can now proceed as in  the proof of Theorem \ref{mainthm}. The complex $\widehat{\qDR}( A/R)$ can be realised as  the cochain complex underlying a cosimplicial $\Lambda$-ring $\hat{U}(A)$, representing the functor $\tilde{X}_{\strat}^q$ of Definition \ref{tildeXqdef} for $X = \Spec A$,
restricted to $(q-1)$-adically complete $\Lambda$-rings $B$. By the consequences of formal \'etaleness, we have
\begin{align*}
 &\Hom_{\Lambda,R}(A,B)\by_{\Hom_{\Lambda,R}(A,B/(q-1))} \Hom_{\Lambda,R}(A,B)\\
 &\cong \Hom_{\Lambda,R}(A,B)\by_{\Hom_{\Lambda,R}( R[x_1, \ldots,x_d],B/(q-1))} \Hom_{\Lambda,R}( R[x_1, \ldots,x_d],B), 
\end{align*}
giving $(\tilde{X}_{\strat}^q)_n \cong \Hom_{\Lambda,R}(A,B)\by_{\Hom_{\Lambda,R}( R[x_1, \ldots,x_d],B)} (\tilde{Y}_{\strat}^q)_n $ for each $n$, where $Y= \Spec  R[x_1, \ldots,x_d]$ and  the fibre product is given via  the projection of $(\tilde{Y}_{\strat}^q)_n$ onto the first factor. 

In particular, this means that $\hat{U}(A)^n$ is the $(q-1)$-adic completion of 
\[
A\ten_{ R[x_1, \ldots,x_d]}(U(R[x_1])^n\ten_{R[q]} \ldots \ten_{R[q]}U(R[x_d])^n), 
\]
where each $U(R[x_i])$ is a copy of the cosimplicial ring $U$ from Proposition \ref{Qcalc}.
This isomorphism respects the cosimplicial operations; note that $\pd^0$ is not linear for the left multiplication by $A$, but is still determined via formal \'etaleness of the framing.

We now define a cosimplicial cochain complex $\tilde{\Omega}^{\bt}(\hat{U}(A))$ by setting $\tilde{\Omega}^{\bt}(\hat{U}(A)^n) $ to be the $(q-1)$-adic completion of 
\begin{align*}
 &(A\ten_{ R[x_1, \ldots,x_d]} (\tilde{\Omega}^*(U(R[x_1])^n)\ten_{R[q]} \ldots \ten_{R[q]}\tilde{\Omega}^*(U(R[x_d])^n)), (q-1)\nabla_q)\\
&\cong (\hat{U}(A)^n\ten_{A^{\ten(n+1)}}(\Omega^*_{A/R})^{\ten(n+1)},(q-1)\nabla_q) ).
\end{align*}
where each  $\tilde{\Omega}^{\bt}(U(R[x_i]))$ is a copy of the complex  $\tilde{\Omega}^{\bt}(U^n)$ from the proof of Theorem \ref{mainthm}. 
Compatibility of this construction with the cosimplicial operations follows because the $\gamma_i$ are  $\Lambda$-ring  homomorphisms.

 The calculations contributing to the proof of  Theorem \ref{mainthm} are still valid after base change, with contracting homotopies giving quasi-isomorphisms 
\[
(\Omega^*_{A/R}\llb q \rrb, (q-1)\nabla_q) \la \Tot N \tilde{\Omega}^{\bt}(\hat{U}(A)^{\bt}) \to N \hat{U}(A)^{\bt}.
 \]
Reduction of this  modulo $(q-1)^2$, or of its d\'ecalage modulo $(q-1)$ (cf. \cite[Proposition 6.12]{BhattMorrowScholze}), replaces $\nabla_q$ with $d$ throughout, removing any dependence on co-ordinates. 
\end{proof}
 
As in \cite[Definition 7.3]{scholzeqdef}, there is a notion of $q$-connection $\nabla_q= (\nabla_{1,q}, \ldots, \nabla_{d,q})$ on a finite projective $A\llb q-1 \rrb$-module $M$, in the form of commuting $R\llb q-1 \rrb$-linear operators $\nabla_{i,q}$ on $M$, with each $\nabla_{i,q}$ satisfying $\nabla_{i,q}(av)= \nabla_{q,x_i}(a)v + \gamma_i(a) \nabla_{i,q}(v)$ for $a \in A, v \in M$.

\begin{definition}
 Given a flat morphism $R \to A$ of $\Lambda$-rings with $X:=\Spec A$,
denote  the forgetful functor $(B,f) \mapsto B$ from  $\hat{\Strat}^q_{A/R}$ to rings by $\sO_{\hat{X}^q,\strat}$. 

There is then a notion of  $\sO_{\hat{X}^q,\strat}$-modules in the category of functors from $\hat{\Strat}^q_{A/R}$ to abelian groups; we will simply refer to these as $\sO_{\hat{X}^q,\strat}$-modules. Given a property $P$ of modules, we will say that an $\sO_{\hat{X}^q,\strat}$-module $\sF$ has the property $P$ if for each $(B,f) \in \hat{\Strat}^q_{A/R}$, the $B$-module $\sF(B,f)$ has property $P$. 

We say that an $\sO_{\hat{X}^q,\strat}$-module $\sF$ is Cartesian if for each morphism $(B,f) \to (B',f')$ in $\hat{\Strat}^q_{A/R}$, the map $\sF(B,f)\ten_BB' \to \sF(B',f')$ is an isomorphism. 

Given an  $\sO_{\hat{X}^q,\strat}$-module $\sF$, we define $\Gamma(\hat{X}^q_{\strat},\sF):= \Lim_{\hat{\Strat}^q_{A/R} }\sF$.
\end{definition}

In \cite[Conjecture 7.5]{scholzeqdef}, Scholze predicted that the category of $q$-connections on finite projective $A\llb q-1 \rrb$-module is independent of co-ordinates on $A$. The following proposition gives the weaker statement  that the category depends only on the $\Lambda$-ring structure on $A$. 

\begin{proposition}\label{qconnprop}
Under the conditions of Theorem \ref{mainthm2}, with $X:=\Spec A$, the  category of finite projective $A\llb q-1 \rrb$-modules $(M,\nabla)$ with $q$-connection is  equivalent to the category of those finite projective $\sO_{\hat{X}^q,\strat}$-modules $\sN$ for which the map 
\[
 \Gamma(\hat{X}^q_{\strat},\sN/(q-1))\ten_A (\sO_{\hat{X}^q,\strat}/(q-1)) \to \sN/(q-1)
\]
 is an isomorphism.
\end{proposition}
\begin{proof}
The restriction on $\sN/(q-1)$ ensures that it is Cartesian; this also implies that $\sN$ is Cartesian, because finite projective modules are flat and $(q-1)$-adically complete. 

Now, the cosimplicial $\Lambda$-ring $\hat{U}(A)$ realising $\widehat{\qDR}( A/R)$  in the proof of Theorem \ref{mainthm2} admits a natural map $A \to \hat{U}(A)/(q-1)$ from the constant cosimplicial diagram. Thus  $\hat{U}(A)$ defines a cosimplicial diagram in $\hat{\Strat}^q_{A/R}$. Since the functor  $\tilde{X}_{\strat}^q$ of Definition \ref{tildeXqdef}
 resolves $X_{\strat}^q$, it follows that the functor $\hat{U}(A) \co \Delta  \to \hat{\Strat}^q_{A/R}$ from the simplex category is initial in the sense of \cite[\S IX.3]{mac}.

In particular, this means that the  category of Cartesian $\sO_{\hat{X}^q,\strat}$-modules $\sN$ is equivalent to the category of Cartesian cosimplicial $\hat{U}(A)$-modules $N$, where the Cartesian condition amounts to saying that the maps $N^m\ten_{\hat{U}(A)^m, \pd^i}\hat{U}(A)^{m+1} \to N^{m+1}$ are all isomorphisms. Setting $M=N^0$, Cartesian $\hat{U}(A)$-modules are equivalent to $\hat{U}(A)^0=A\llb q-1 \rrb$-modules $M$ with isomorphisms $\Delta \co (\pd^1)^*M \cong (\pd^0)^*M$ satisfying the cocycle condition $\pd^1\Delta = (\pd^0\Delta) \circ (\pd^2\Delta) \co (\pd^2\pd^0)^*M \to  (\pd^0\pd^0)^*M$.

The map $\Delta$ is determined by its restriction to $M$, so using the basis for $U^1$ from Lemma \ref{subringbasis}, and taking $v \in M$,  we have
\[
 \Delta(v) = \sum_{\uline{k} \in \N_0^d} \pd^0(\Delta_{\uline{k}}(v)) \lambda^{k_1}(\tfrac{\pd^1x_1-\pd^0x_1}{q-1})\cdots \lambda^{k_d}(\tfrac{\pd^1x_d-\pd^0x_d}{q-1})
\]
for $R\llb q-1 \rrb$-linear endomorphisms $\Delta_{\uline{k}}$ of $M$. Since $\lambda_t(a+b)= \lambda_t(a)\lambda_t(b)$, the cocycle condition becomes $\Delta_{\uline{j}+\uline{k}}= \Delta_{\uline{j}}\circ \Delta_{\uline{k}}$, meaning $\Delta$  is determined by the operators $\Delta_{e_i}$ at the basis vectors, which must moreover commute. 

Linearity of $\Delta$ with respect to $\hat{U}(A)^1$ then reduces to the condition that $\Delta( av) = \pd^1(a)\Delta(v)$ for $a \in A$, $v \in M$. Writing $A$ for $\pd^0A$ and $h_i^{[k]}:= \lambda^{k}(\tfrac{\pd^1x_i-\pd^0x_i}{q-1})$, 
 the ideal $J:=(h_i^{[\ge 2]},h_ih_j)_{i \ne j}$ satisfies $U^1 = A \oplus \bigoplus_i A h_i \oplus J$. The proof of Theorem \ref{mainthm2} gives $\pd^1(a) \equiv a+ (q-1)\sum_i \nabla_{q,x_i}(a)h_i \mod J$, and in $U^1/J$ we have $[h_i]^2 \equiv x_i[h_i]$. Comparing coefficients of $h_i$ in the equation $\Delta( av) \equiv \pd^1(a)\Delta(v) \mod J$ then gives 
\begin{align*}
  \Delta_{e_i}(av) &= (q-1) \nabla_{q,x_i}(a) v + a\Delta_{e_i}(v) + (q-1)x_i \nabla_{q,x_i}(a) \Delta_{e_i}(v)\\
 &=  (q-1) \nabla_{q,x_i}(a) v + \gamma_i(a) \Delta_{e_i}(v).
 \end{align*}

Finally, note that the condition that $\sN/(q-1)$ be the pullback of an  $A$-module (necessarily $\Gamma(\hat{X}^q_{\strat},\sN/(q-1)) $)  is equivalent to saying that $\pd^0_N \equiv \pd^1_N \mod (q-1)$, or that $(q-1)$ divides $ \Delta_{\uline{k}}$ whenever $\uline{k}\ne 0$. In particular, $(q-1)$ divides $\Delta_{e_i}$, and setting $\nabla_{i,q}:=  (q-1)^{-1}\Delta_{e_i}$ gives a $q$-connection $(\nabla_{i,q})_{1\le i\le d}$ on $M=N^0$ uniquely determining $\Delta$. 

The inverse construction is given by
$
 \Delta_{\uline{k}}= (q-1)^{\sum k_i} \nabla_{1,q}^{k_1}\circ \cdots \circ \nabla_{d,q}^{k_d}. 
$ 
\end{proof}

\section{Comparisons for $\Lambda_P$-rings}

Since very few \'etale maps $R[x_1, \ldots,x_d]\to A$ give rise to $\Lambda$-ring structures on $A$,  Theorem \ref{mainthm2} is fairly limited in its scope for applications. We now show how the construction of $\widehat{\qDR}$ and the  comparison quasi-isomorphism survive  when we  weaken the $\Lambda$-ring structure by discarding Adams operations at invertible primes.

\subsection{$q$-cohomology for $\Lambda_P$-rings}

Our earlier constructions for $\Lambda$-rings all  carry over to $\Lambda_P$-rings, as follows.

\begin{definition}
Given a set $P$ of primes, we define a  $\Lambda_P$-ring $A$ to be a  $\Lambda_{\Z,P}$-ring in the sense of \cite{borgerBasicGeomI}. This means that it is a coalgebra in commutative rings for the comonad given by the functor $W^{(P)}$ of $P$-typical Witt vectors. When a commutative ring $A$ is flat over $\Z$, giving a $\Lambda_P$-ring structure on $A$ is equivalent to giving commuting Adams operations $\Psi^p$ for all $p \in P$, with $\Psi^p(a) \equiv a^p \mod p$ for all $a$.
\end{definition}
Thus when $P$ is the set of all primes, a $\Lambda_P$-ring is just a $\Lambda$-ring; a $\Lambda_{\emptyset}$-ring is just a commutative ring; for a single prime $p$, we write $\Lambda_p:= \Lambda_{\{p\}}$, and note that a $\Lambda_p$-ring is a $\delta$-ring in the sense of \cite{joyaldeltaWitt}.

\begin{definition}
 Given a $\Lambda_P$-ring $R$, say that $A$ is a $\Lambda_P$-ring over $R$ if it is a $\Lambda_P$-ring equipped with a morphism $R \to A$ of $\Lambda_P$-rings. We say that $A$ is a flat $\Lambda_P$-ring over $R$ if $A$ is flat as a module over the commutative ring   underlying $R$.
\end{definition}

\begin{definition}\label{qDRdefP}
Given a morphism $R \to A$ of $\Lambda_P$-rings, we define the category $\Strat^{q,P}_{A/R}$ to consist of flat $\Lambda_P$-rings $B$ over $R[q]$ equipped with a compatible  morphism $A \to B/(q-1)$, such that 
 the map $A \to   B/(q-1)$ admits a lift to $B$. We define the category $\hat{\Strat}^{q,P}_{A/R} \subset \Strat^q_{A/R}$ to consist of those objects which are $(q-1)$-adically complete.
\end{definition}
More concisely, $\Strat^{q,P}_{A/R}$  (resp. $\hat{\Strat}^{q,P}_{A/R}$) is the Grothendieck construction of the functor $(\Spec A)_{\strat}^{q,P}$ (resp. $\widehat{(\Spec A)}_{\strat}^{q,P}$) given by
\[
 B \mapsto \im(\Hom_{\Lambda_P,R}(A,B)\to \Hom_{\Lambda_P,R}(A,B/(q-1))) 
\]
on the category of flat $\Lambda_P$-rings (resp. $(q-1)$-adically complete flat $\Lambda_P$-rings) over $R[q]$. 

\begin{definition}
 Given a flat morphism $R \to A$ of $\Lambda_P$-rings, define $\qDR_P(A/R)$ to be the cochain complex of $R[q]$-modules given by taking the homotopy limit of
 the functor
\begin{align*}
 \Strat^{q,P}_{A/R} &\to \Ch(R[q])\\
B &\mapsto B.
\end{align*}
Define $\widehat{\qDR}_P(A/R)$ to be the cochain complex of $R\llbracket q-1 \rrbracket$-modules given by the corresponding homotopy limit over $\hat{\Strat}^{q,P}_{A/R}$.

For $p \in P$, the cochain complex $\qDR_P(A/R)$ naturally carries $(R[q], \Psi^p)$-semilinear operations $\Psi^p$ coming from the morphisms $\Psi^p \co B\ten_{R[q], \Psi^p}R[q] \to B$ of $R[q]$-modules, for $B \in \Strat^{q,P}_{A/R}$.
\end{definition}

Thus when $P$ is the set of all primes, we have $\qDR_P(A/R)= \qDR(A/R)$. At the other extreme, for $A$ smooth,  $\widehat{\qDR}_{\emptyset}(A/R)$ is the Rees construction of the Hodge filtration on the infinitesimal cohomology complex \cite{Gr} of $A$ over $R$ , with formal variable $(q-1)$. In more detail, there is a decreasing  filtration $F$ of $\sO_{\inf}$  given by powers of the augmentation ideal of  $\sO_{\inf} \to \sO_{\Zar}$ (with $F^{\nu}\sO_{\inf}=\sO_{\inf}$ for $\nu \le 0$), and then 
\[
\widehat{\qDR}_{\emptyset}(A/R) \simeq \prod_{\nu\in \Z} (q-1)^{-\nu}\oR\Gamma(\Spec A, F^{\nu}\sO_{\inf}).
\]

\begin{lemma}
 For a set $P$ of primes, the forgetful functor from $\Lambda$-rings to $\Lambda_P$-rings has a right adjoint $W^{(\notin P)}$. There is a canonical ghost component morphism 
\[
 W^{(\notin P)}(B) \to \prod_{\substack{n \in \N:\\ (n,p)=1 ~\forall p \in P}} B,
\]
which is an isomorphism when $P$ contains all the residue characteristics of 
$B$.
\end{lemma}
\begin{proof}
 Existence of a right adjoint follows from the comonadic definitions of $\Lambda$-rings and $\Lambda_P$-rings. The ghost component morphism is given by taking the Adams operations $\Psi^n$ 
coming from the $\Lambda$-ring structure on $W^{(\notin P)}(B)$, followed by projection to $B$. When $P$ contains all the residue characteristics of $B$, a $\Lambda$-ring structure is the same as a $\Lambda_P$-ring structure with compatible commuting Adams operations for all primes not in $P$, leading to the description above.
\end{proof}

Note that the big Witt vector functor $W$ on commutative rings thus factorises as $W= W^{(\notin P)} \circ W^{(P)}$, for $W^{(P)}$ the $P$-typical Witt vectors.

\begin{proposition}\label{Ponlyprop}
  Given a morphism $R \to A$ of $\Lambda$-rings, and a set $P$ of primes, there are natural maps
\begin{align*}
 \qDR_P(A/R) &\to \qDR(A/R), &
\widehat{\qDR}_P(A/R) &\to \widehat{\qDR}(A/R),
\end{align*}
and the latter map is a quasi-isomorphism when $P$ contains all the residue characteristics of $A$.
\end{proposition}
\begin{proof}
We have functors
 \begin{align*}
  (\Spec A)_{\strat}^q\circ W^{(\notin P)} \co B &\mapsto \im(\Hom_{\Lambda,R}(A,W^{(\notin P)} B)\to \Hom_{\Lambda,R}(A,(W^{(\notin P)} B)/(q-1)))\\
(\Spec A)_{\strat}^{q,P} \co  B &\mapsto \im(\Hom_{\Lambda_P,R}(A,B)\to \Hom_{\Lambda_P,R}(A,B/(q-1)))
 \end{align*}
on the category of flat $\Lambda_P$-rings over  $R[q]$. There is an obvious map 
\[
 (W^{(\notin P)} B)/(q-1) \to W^{(\notin P)}(B/(q-1)),
\]
 and hence a natural transformation $(\Spec A)_{\strat}^q\circ W^{(\notin P)} \to (\Spec A)_{\strat}^{q,P}$, which induces the morphism $\qDR_P(A/R) \to \qDR(A/R)$ on cohomology.

When $P$ contains all the residue characteristics of $A$, the map $ (W^{(\notin P)} B)/(q-1) \to W^{(\notin P)}(B/(q-1))$ is just
\[
  \prod_{\substack{n \in \N:\\ (n,p)=1 ~\forall p \in P}} B/(q^n-1) \to  \prod_{\substack{n \in \N:\\ (n,p)=1 ~\forall p \in P}} B/(q-1),
\]
since the morphism $R[q] \to W^{(\notin P)} B$ is given by Adams operations, with $\Psi^n(q-1)= q^n-1$. 

We have $(q^n-1)=(q-1)[n]_q$, and $[n]_q$ is a unit in $\Z[\tfrac{1}{n}]\llbracket q-1 \rrbracket$, hence a unit in $B$ when $n$ is coprime to the residue characteristics. Thus the map $ (W^{(\notin P)} B)/(q-1) \to W^{(\notin P)}(B/(q-1))$ gives an isomorphism whenever $B$ is $(q-1)$-adically complete and admits a map from $A$, so the transformation $(\Spec A)_{\strat}^q\circ W^{(\notin P)} \to (\Spec A)_{\strat}^{q,P}$ is a natural  isomorphism on the category of flat $(q-1)$-adically complete $\Lambda_P$-rings over  $R[q]$, and hence
$
 \widehat{\qDR}_P(A/R) \xra{\simeq} \widehat{\qDR}(A/R)
$.
\end{proof}

\begin{remark}\label{cosimplicialrmkP}
 Remark \ref{cosimplicialrmk} shows that $\qDR(A/R)$ can naturally be promoted to a cosimplicial $\Lambda$-ring, and the same reasoning promotes $\qDR_P(A/R)$ to a cosimplicial $\Lambda_P$-ring.
The proof of Proposition \ref{Ponlyprop} then ensures that the map $\qDR_P(A/R) \to \qDR(A/R)$ is naturally a morphism of cosimplicial $\Lambda_P$-rings,  
\end{remark}

Over $\Z[\{\frac{1}{p}\co p \in P\}]$, every $\Lambda_P$-ring can be canonically made into a $\Lambda$-ring, by setting all the additional Adams operations to be the identity. However, this observation is of limited use in  establishing functoriality of $q$-de Rham cohomology, because the resulting $\Lambda$-ring structure will not satisfy the conditions of Theorem \ref{mainthm2}. We now give a more general result which does allow for meaningful comparisons.

\begin{theorem}\label{mainthm3}
 If $R$ is a flat $\Lambda_P$-ring over $\Z$ and $\boxempty\co R[x_1, \ldots,x_d]\to A$ is a formally \'etale map of $\Lambda_P$-rings,  the elements $x_i$ having rank $1$, then there are zigzags of $R\llbracket q-1\rrbracket$-linear quasi-isomorphisms
\begin{align*}
\widehat{\qDR_P}( A/R) &\simeq  (\Omega^*_{A/R}\llb q -1\rrb, (q-1)\nabla_q), &
 \oL\eta_{(q-1)} \widehat{\qDR}_P( A/R) &\simeq \widehat{q\mbox{-}\Omega}^{\bt}_{A/R, \boxempty}.
\end{align*}
whenever $P$ contains all the residue characteristics of $A$.
\end{theorem}
\begin{proof}
The key observation to make is that formally \'etale maps have a unique lifting property with respect to nilpotent extensions of flat $\Lambda_P$-rings, because the Adams operations must also lift uniquely.  In particular, this means that the operations $\gamma_i$ featuring in the definition of $q$-de Rham cohomology are necessarily endomorphisms of $A$ as a $\Lambda_P$-ring.

Similarly to Theorem \ref{mainthm2},  $\widehat{\qDR}_P( A/R)$ is calculated using a cosimplicial $\Lambda_P$-ring given in level $n$ by the $(q-1)$-adic completion $\hat{U}_{P,A}^{\bt}$ of the $\Lambda_P$-ring  over $R[q]$ generated by $A^{\ten_R(n+1)}[q]$ and $(q-1)^{-1}\ker( A^{\ten_R(n+1)}\to A)[q]$. The observation above shows that $\hat{U}_{P,A}^n \cong \hat{U}_{P,R[x_1, \ldots,x_d]}^n\hten_{ R[x_1, \ldots,x_d]}A$, changing base along $\boxempty$ applied to the first factor. 

As in Proposition \ref{Ponlyprop}, $\hat{U}_{P,R[x_1, \ldots,x_d]}^{\bt}$ is just the $(q-1)$-adic completion of the complex $U^{\bt}$ from Proposition \ref{Qcalc}.
Further application of the key observation above then allows us to  adapt the constructions of Theorem \ref{mainthm}, giving the desired quasi-isomorphisms.

\end{proof}

\subsection{Cartier isomorphisms in mixed characteristic}\label{cartiersn}

In \cite[Conjecture 7.1]{scholzeqdef}, Scholze predicted that $\widehat{q\mbox{-}\Omega}^{\bt}_{A/R, \boxempty}$ is a functorial invariant of the $R$-algebra $A$, independent of the choice of framing, so extends to all smooth schemes. Theorem \ref{mainthm3} shows that $\widehat{q\mbox{-}\Omega}^{\bt}_{A/R, \boxempty}$ is functorial invariant of the $\Lambda_P$-ring $A$ over $R$.

 The only setting in which Theorem \ref{mainthm3} leads to results  close to Scholze's conjecture is when $R=W^{(p)}(k)$, the $p$-typical Witt vectors of a  perfect field of characteristic $p$, and $A=\Lim_n A_n$ is a formal deformation of a smooth $k$-algebra $A_0$. 
Then any formally \'etale morphism $W^{(p)}(k)[x_1, \ldots,x_d]\to A$ of topological rings gives rise to a unique compatible lift $\Psi^p$ of absolute Frobenius on $A$ with $\Psi^p(x_i)=x_i^p$, so gives $A$ the structure of a topological $\Lambda_p$-ring. The framing still affects the choice of $\Lambda_p$-ring structure, but at least such a structure is guaranteed to exist, giving rise to a complex $\qDR_P(A/R)^{\wedge_p}:= \oR\Lim_n \qDR_p(A/R)\ten_R^{\oL}R_n$ depending only on the choice of $\Psi^p$, where $R_n=W_n^{(p)}(k)$. 


Our constructions now allow us to globalise the quasi-isomorphism 
\[
 (\widehat{q\mbox{-}\Omega}^{\bt}_{A/R, \boxempty})^{\wedge_p}/[p]_q \simeq (\Omega^*_{A/R})^{\wedge_p}\llbracket q-1 \rrbracket/[p]_q
\]
of \cite[Proposition 3.4]{scholzeqdef}, where $\Omega^*_{A/R}$ denotes the complex $A \xra{0} \Omega^1_{A/R} \xra{0}  \Omega^2_{A/R} \xra{0}\ldots $.

\begin{lemma}\label{PsiqDRlemma}
 Under the quasi-isomorphism $\widehat{\qDR}_p( A/R) \simeq (\Omega^*_{A/R}\llbracket q-1 \rrbracket, (q-1)\nabla_q)$ from Theorem \ref{mainthm3}, the semilinear Adams operation $\Psi^{p}$  on $ \widehat{\qDR}_p( A/R) $ described in Definition \ref{qDRdef} corresponds to the operation on $\Omega^*_{A/R}\llbracket q-1 \rrbracket$ given by setting
\[
\Psi^p (a dx_{i_1}\wedge \ldots\wedge dx_{i_m}):= \Psi^p(a) x_{i_1}^{p-1}\ldots x_{i_m}^{p-1} dx_i\wedge \ldots\wedge dx_{i_m}.
\]
for $a \in A\llb q-1 \rrb$.
\end{lemma}
\begin{proof}
Just observe that this expression defines a chain map on $(\Omega^*_{A/R}\llbracket q-1 \rrbracket, (q-1)\nabla_q)$ (for instance $\Psi^p((q-1)\nabla_qx_i)= (q^p-1)\Psi^p(dx)= (q-1)\nabla_qx_i^p$), and that the quasi-isomorphisms  in the proof of Theorem \ref{mainthm2} commute with these operations.
\end{proof}

As in \cite[\S 4]{scholzeqdef}, we refer to formal schemes over $W^{(p)}(k)$  as smooth if they are flat deformations of smooth schemes over $k$. We refer to morphisms of such schemes as \'etale if they are flat deformations of \'etale morphisms over $k$. 

\begin{proposition}\label{cartierprop}
 Take  a smooth formal scheme $\fX$ over $R=W^{(p)}(k)$ equipped with a lift $\Psi^p$  of Frobenius which \'etale locally admits  co-ordinates $\{x_i\}_i$ as above with $\Psi^p(x_i)=x_i^p$. Then there is a global  quasi-isomorphism
\[
C_q^{-1}\co  (\Omega^*_{\fX/R})^{\wedge_p}\llbracket q-1 \rrbracket/[p]_q\to (\oL\eta_{(q-1)}\widehat{\qDR}_p(\sO_{\fX}/R))^{\wedge_p}/[p]_q
\]
in the derived category of \'etale sheaves on $\fX$.  
\end{proposition}
\begin{proof}
The unique lifting property of formally \'etale morphisms ensures that each affine formal scheme $\fU$ \'etale over $\fX$ has a unique lift $\Psi^p|_{\fU}$ of Frobenius compatible with the given operation $\Psi^p$ on $\fX$. Functoriality of the construction $\qDR_p$ for rings with Frobenius lifts thus gives us an \'etale presheaf $\widehat{\qDR}_p(\sO_{\fX}/R)^{\wedge_p}$ of complexes on $\fX$.  As in Definition \ref{qDRdef},  the Adams operation $\Psi^p$ on  $\sO_{\fX}$  then extends to $(R\llbracket q-1 \rrbracket, \Psi^p)$-semilinear maps
\begin{align*}
\Psi^p \co  \qDR_p(\sO_{\fX}/R)^{\wedge_p}&\to \qDR_p(\sO_{\fX}/R)^{\wedge_p}\\
   \qDR_p(\sO_{\fX}/R)^{\wedge_p}/(q-1)&\to \qDR_p(\sO_{\fX}/R)^{\wedge_p}/(q^p-1),
\end{align*}
and thus,  denoting good truncation by $\tau$,
\[
  (q-1)^i\Psi^p \co \tau^{\le i}(\qDR_p(\sO_{\fX}/R)^{\wedge_p}/(q-1))\to (\oL\eta_{(q-1)}\widehat{\qDR}_p(\sO_{\fX}/R)^{\wedge_p})/[p]_q;
\]
the left-hand side is quasi-isomorphic to  $\bigoplus_{j \le i} (\Omega^j_{\sO_{\fX}/R})^{\wedge_p}[-j]$ by Theorem \ref{mainthm2}. 

Extending the construction  $R[q]$-linearly and restricting to top summands therefore gives us the global map $C_q^{-1}$.  For a local choice of framing, Lemma \ref{PsiqDRlemma}  gives equivalences
\[
(q-1)^i\Psi^p \simeq \sum_{j \le i}  (q-1)^{i-j} (\tilde{C}^{-1})^j
\]
for Scholze's locally defined lifts $(\tilde{C}^{-1})^j\co (\Omega^j_{A/R})^{\wedge_p}[-j] \to (\widehat{q\mbox{-}\Omega}^{\bt}_{A/R, \boxempty})^{\wedge_p}/[p]_q$ of the Cartier quasi-isomorphism. The local calculation of \cite[Proposition 3.4]{scholzeqdef} then ensures that $C_q^{-1}$  is a quasi-isomorphism.
\end{proof}

\section{Functoriality via analogues of de Rham--Witt cohomology}\label{DRWsn}

In order to obtain a cohomology theory for smooth commutative rings rather than for $\Lambda_P$-rings, we now consider $q$-analogues of de Rham--Witt cohomology. Our starting point is to observe that if we allow roots of $q$, we can extend the Jackson differential to fractional powers of $x$ by the formula 
\[
 \nabla_q (x^{m/n}) = \frac{q^{m/n} -1}{q-1} x^{m/n}d\log x,
\]
where $d \log x= x^{-1}dx$,
so terms such as $[n]_{q^{1/n}}x^{m/n}$ have integral derivative, where $[n]_{q^{1/n}}= \frac{q-1}{q^{1/n}-1}$.

\subsection{Motivation}\label{motivnsn}

\begin{definition}
 Given a $\Lambda_P$-ring $B$, define  $\Psi^{1/P^{\infty}}B$ to be the smallest  $\Lambda_P$-ring which is equipped with a morphism from  $B$  and for which the Adams operations are automorphisms.
\end{definition}
In the case $P= \{p\}$, the $\Lambda_p$-ring $\Psi^{1/p^{\infty}}B$ is thus the colimit of the diagram 
\[
B \xra{\Psi^{p}}  B \xra{\Psi^{p}} B \xra{\Psi^{p}} \ldots.
\]

By Remark \ref{cosimplicialrmkP}, $\widehat{\qDR}_p( A/R)$ naturally underlies a cosimplicial $\Lambda_p$-ring, so applying $\Psi^{1/p^{\infty}}$ levelwise gives another cosimplicial $\Lambda_p$-ring. For the Adams operation $\Psi^p$ of Definition \ref{qDRdefP}, the underlying cochain complex is just $\Psi^{1/p^{\infty}}\widehat{\qDR}_p( A/R):= \LLim_{\Psi^p} \widehat{\qDR}_p( A/R)$.
As an immediate consequence of Lemma \ref{PsiqDRlemma}, we have:

\begin{lemma}\label{qDRWlemma1}
If $R$ is a flat $\Lambda_p$-ring over $\Z_{(p)}$ with $\Psi^p$ an isomorphism, then $\Psi^{1/p^{\infty}}\widehat{qDR}_p( R[x]/R)$ is quasi-isomorphic to the  complex 
\[
 ( R[x^{1/p^{\infty}}, q^{1/p^{\infty}}] \xra{(q-1)\nabla_q} (x^{1/p^{\infty}})R[x^{1/p^{\infty}}, q^{1/p^{\infty}}] d\!\log x)^{\wedge_{(q-1)}},
\]
so
 the d\'ecalage
 $\oL\eta_{(q-1)}\Psi^{1/p^{\infty}}\widehat{\qDR}_p( R[x]/R)$
and the complex 
\begin{align*}
 & \{ a \in R[x^{1/p^{\infty}}, q^{1/p^{\infty}}] ~:~ \nabla_qa \in R[x^{1/p^{\infty}}, q^{1/p^{\infty}}]d\log x\} \\
 &\xra{\nabla_q}
(x^{1/p^{\infty}})R[x^{1/p^{\infty}}, q^{1/p^{\infty}}] d\!\log x.
\end{align*}
are quasi-isomorphic after $(q-1)$-adic completion. 
\end{lemma} 
Thus in level $0$ (resp.\ level $1$), $\oL\eta_{(q-1)}\Psi^{1/p^{\infty}}\widehat{\qDR}( R[x]/R)$ is spanned by elements of the form $[p^n]_{q^{1/p^n}}x^{m/p^n}$ (resp.\ $x^{m/p^n} d\log x$),  so setting $q^{1/p^{\infty}}=1$ gives a complex whose $p$-adic completion is the $p$-typical de Rham--Witt complex. 

\begin{lemma}\label{perfectioncoholemma}
Let  $R$ and $A$ be flat $p$-adically complete $\Lambda_p$-algebras over $\Z_p$, with $\Psi^p$ an isomorphism on $R$. For  elements $x_i$ of rank $1$, take a map $\boxempty\co R[x_1, \ldots,x_d]^{\wedge_p}\to A$ of $\Lambda_p$-rings which is a flat $p$-adic deformation of an \'etale map.
Then the map
\[
(R[q^{1/p^{\infty}}]\ten_{R[q]} \oL\eta_{(q-1)} \widehat{\qDR}_p( A/R))^{\wedge_p} \to \oL\eta_{(q-1)} (\Psi^{1/p^{\infty}}\widehat{\qDR}_p (A/R))^{\wedge_p}
\]
is a quasi-isomorphism.
\end{lemma}
\begin{proof}
The map $\Psi^p \co A\ten_{R[x_1, \ldots,x_d]}R[x_1^{1/p}, \ldots,x_d^{1/p}] \to A$ becomes an isomorphism on $p$-adic completion, because  $\boxempty$ is flat and we have an isomorphism  modulo  $p$.
  Thus 
\[
\Psi^{1/p^{\infty}}A \cong A[x_1^{1/p^{\infty}}, \ldots,x_d^{1/p^{\infty}}]^{\wedge_p}:= (A\ten_{R[x_1, \ldots,x_d]}R[x_1^{1/p^{\infty}}, \ldots,x_d^{1/p^{\infty}}])^{\wedge_p}
\]

Combined  with the calculation of Lemma \ref{PsiqDRlemma}, this  gives us a quasi-isomorphism between $(\Psi^{1/p^{\infty}}\widehat{\qDR}_p(A/R))^{\wedge_p}$ and the $(p,q-1)$-adic completion of 
\[
(\bigoplus_I \bigoplus_{\alpha} A\llb q-1\rrb x_1^{\alpha_1}\ldots x_d^{\alpha_d} dx^I[-|I|], (q-1)\nabla_q),
\]
where $I$ ranges over finite subsets of $\{1, \ldots, d\}$ and  $\alpha$ ranges over elements of $p^{-\infty}\Z^d$ with $0 \le \alpha_i <1$ if $i \notin I$ and $-1< \alpha_i \le 0$ if $i \in I$.

We then observe that the contributions to the d\'ecalage $\eta_{(q-1)}$ from terms with $\alpha \ne 0$ must be acyclic, via a
contracting homotopy defined by the restriction to $\eta_{(q-1)}$ of the $q$-integration map
\[
 f x_1^{\alpha_1}\ldots x_d^{\alpha_d} dx^I \mapsto f x_1^{\alpha_1}\ldots  x_d^{\alpha_d}\sum_{i \in I}  \pm x_i [\alpha_i]_q^{-1} dx^{(I\setminus i)},
\]
where $[\frac{m}{p^n}]_q^{-1}= [m]_{q^{1/p^n}}^{-1}[p^n]_{q^{1/p^n}}$ for $m$ coprime to $p$, noting that $[m]_{q^{1/p^n}}$ is a unit in $\Z[q^{1/p^{\infty}}]^{\wedge_{(p,q-1)}}$.
\end{proof}

\begin{remark}\label{Wittrmk}
 The endomorphism given on $\Psi^{1/P^{\infty}}\widehat{\qDR}_P(A/R)$ by 
\[
 a \mapsto \Psi^{1/n}([n]_qa)= [n]_{q^{1/n}}\Psi^{1/n}a
\]
 descends to an endomorphism  of  $\H^0(\Psi^{1/P^{\infty}}\widehat{\qDR}_P(A/R)/(q-1))$, which we may denote by $V_n$ because it mimics Verschiebung in the sense that $\Psi^{n}V_n= n\cdot\id$ (since $[n]_q \equiv n \mod (q-1)$). For $A$ smooth over $\Z$, we then have
\begin{align*}
\H^0(\Psi^{1/P^{\infty}}\widehat{\qDR}_P(A/\Z)/(q-1))/(V_p~:~p \in P) &\cong A[q^{1/P^{\infty}}]/([p]_{q^{1/p}}~:~p \in P)\\ &\cong A[\zeta_{P^{\infty}}],
\end{align*}
for $\zeta_n$ a primitive $n$th root of unity.

By adjunction, this gives  an injective map 
\[
 \H^0(\Psi^{1/P^{\infty}}\widehat{\qDR}_P(A/\Z)/(q-1)) \into W^{(P)}A[\zeta_{P^{\infty}}]
\]
of $\Lambda_P$-rings, which becomes an isomorphism on completing $\Psi^{1/P^{\infty}}\widehat{\qDR}(A/\Z)$ with respect to the system $\{([n]_{q^{1/n}})\}_{n\in P^{\infty}}$ of ideals, where we write $P^{\infty}$ for the set of integers whose prime factors are all in $P$. This implies that the cokernel is annihilated by all elements of $(q^{1/P^{\infty}}-1)$, so leads us to consider almost mathematics as in \cite{gabberalmost}. 
\end{remark}


\subsection{Almost isomorphisms}
From now on, we consider only the case  $P=\{p\}$.
Combined with Lemma \ref{perfectioncoholemma}, Remark \ref{Wittrmk} allows us to regard $\oL\eta_{(q-1)}\Psi^{1/p^{\infty}}\widehat{\qDR}_p(A/\Z_p)^{\wedge_p}$ as being almost a $q^{1/p^{\infty}}$-analogue of $p$-typical de Rham--Witt cohomology.

The ideal $(q^{1/p^{\infty}}-1)^{\wedge_{(p,q-1)}}= \ker(\Z[q^{1/p^{\infty}}]^{\wedge_{(p,q-1)}} \to \Z_p)$ is equal to the $p$-adic completion of its square, since we may write it as the kernel $W^{(p)}(\m)$ of $W^{(p)}(\bF_p[q^{1/p^{\infty}}]^{\wedge_{(q-1)}})\to W^{(p)}(\bF_p)$, for the idempotent maximal ideal $\m = ( (q-1)^{1/p^{\infty}})^{\wedge_{(q-1)}}$ in $\bF_p[q^{1/p^{\infty}}]^{\wedge_{(q-1)}}$. If we set $h^{1/p^n}$ to be the Teichm\"uller element 
\[
 [q^{1/p^n}-1] = \lim_{r \to \infty} (q^{1/p^{nr}}-1)^{p^r} \in \Z[q^{1/p^{\infty}}]^{\wedge_{(p,q-1)}},
\]
 then $W^{(p)}(\m)=(h^{1/p^{\infty}})^{\wedge_{(p,h)}}$. Although $W^{(p)}(\m)/p^n$   is not maximal in $\Z[h^{1/p^{\infty}}]^{\wedge_{(h)}}/p^n$,  it is idempotent and flat, so gives a  basic setup in the sense of \cite[2.1.1]{gabberalmost}. We thus regard the pair $(\Z[q^{1/p^{\infty}}]^{\wedge_{(p,q-1)}},W{(p)}(\m) )$ as an inverse system of basic setups for almost ring theory.

We then follow the terminology and notation of \cite{gabberalmost}, studying $p$-adically complete  $(\Z[q^{1/p^{\infty}}]^{\wedge_{(p,q-1)}})^a$-modules (almost $\Z[q^{1/p^{\infty}}]^{\wedge_{(p,q-1)}}$-modules) given by localising at almost isomorphisms, the maps whose  kernel and cokernel are $W^{(p)}(\m)$-torsion. 

\begin{definition}\label{almosteltsdef}
 The obvious functor $(-)^a$ from modules to almost modules has a right adjoint $(-)_*$, given by  $N_*:=\Hom_{\Z[q^{1/p^{\infty}}]^{\wedge_{(p,q-1)}}}(W^{(p)}(\m) ,N)$, the module of almost elements. 
\end{definition}
Since the counit $(M_*)^a\to M$ of the adjunction is an (almost) isomorphism, we may also regard almost modules as a full subcategory of the category of modules, consisting of those $M$ for which the natural map $M \to (M^a)_*$ is an isomorphism. We can define $p$-adically complete  $(\Z[q^{1/p^{\infty}}]^{\wedge_{(p,q-1)}})^a$-algebras similarly,  forming a full subcategory of $\Z[q^{1/p^{\infty}}]^{\wedge_{(p,q-1)}}$-algebras. 


\subsection{Perfectoid algebras}

We now relate Scholze's perfectoid algebras to a class of $\Lambda_p$-rings, by factorising  the tilting equivalence. For simplicity, we work over $\Z[\zeta_{p^{\infty}}]^{\wedge_p}$, although Lemma \ref{cfperfectoid} has natural analogues over the ring $K^o \subset K$ of power-bounded elements of any perfectoid field $K$ in the sense of \cite{scholzePerfectoidSpaces}.

\begin{definition}
Define Fontaine's period ring functor $\sA_{\inf}$ from commutative rings to $\Lambda_p$-rings by  $\sA_{\inf}(C):= \Lim_{\Psi^p}W^{(p)} (C)$. 
\end{definition}

\begin{definition}\label{perfectoidLambdapdef}
 Define a perfectoid $\Lambda_p$-ring to be a flat $p$-adically complete $\Lambda_p$-algebra over $\Z_p$, on which the Adams operation $\Psi^p$ is an isomorphism. 

By analogy with \cite[Notation 1.4]{bhattDirectSummand}, we say that a  perfectoid $\Lambda_p$-ring over  $\Z[q^{1/p^{\infty}}]^{\wedge_{(p,q-1)}}$ is integral if the morphism 
$
 B \to B_*
$
of Definition \ref{almosteltsdef} is an isomorphism. 
\end{definition}

\begin{lemma}\label{cfperfectoid}
We have equivalences of categories
\[
 \xymatrix{
  \text{perfectoid almost }\Z[\zeta_{p^{\infty}}]^{\wedge_p} \text{-algebras} \ar@<1ex>[d]^{\sA_{\inf}(-)_*} \\
\text{integral perfectoid }\Lambda_p\text{-rings over } \Z[q^{1/p^{\infty}}]^{\wedge_{(p,q-1)}}\ar@<1ex>[u]^{-/[p]_{q^{1/p}} } \ar@<-1ex>[d]_{-/p}\\ 
\text{perfectoid almost }\bF_p[q^{1/p^{\infty}}]^{\wedge_{(q-1)}}\text{-algebras}.\ar@<-1ex>[u]_{W^{(p)}(-)_*}
 }
\]
\end{lemma}
\begin{proof}
A perfectoid $\Lambda_p$-ring $B$ is a deformation of the perfect $\bF_p$-algebra $B/p$. As in \cite[Proposition 5.13]{scholzePerfectoidSpaces},  a perfect $\bF_p$-algebra $C$ has  a unique deformation  $W^{(p)}(C)$ over $\Z_p$, to which Frobenius must lift uniquely; this shows that $W^{(p)}$ gives an equivalence between perfect $\bF_p$-algebras and perfectoid $\Lambda_p$-rings. 
To obtain the bottom equivalence of the diagram, we will show that the functor $W^{(p)}$ commutes with the respective functors $C \mapsto C_*$ of almost elements, then appeal to the tilting equivalence.

Because the idempotent ideals of the basic setups in each of our three categories are generated by the rank $1$ elements $h^{p^{-n}}$ constructed before Definition \ref{almosteltsdef}, we can write $C_*= \bigcap_n h^{-p^{-n}}C$ in each setting. For a Teichm\"uller element $[c] \in W^{(p)}(C)$, the standard isomorphism $W^{(p)}(C)\cong C^{\N_0}$ of sets gives an isomorphism $[c]W^{(p)}(C)\cong \prod_{m \ge 0} c^{p^m}C$. Thus the natural map $ W^{(p)}(C)_* \to  W^{(p)}(C_*)$ of $\Lambda_p$-rings is an isomorphism, since
\[
 W^{(p)}(C)_* \cong \bigcap_{n\ge 0} \prod_{m\ge 0}  h^{-p^{m-n}}C \cong \prod_{m \ge 0} C_* \cong W^{(p)}(C_*),
\]
and taking inverse limits with respect to $\Psi^p$ gives $\sA_{\inf}(C)_*\cong \sA_{\inf}(C_*)$ as well.

Next, we observe that since $B:=\sA_{\inf} (C)$ is a perfectoid  $\Lambda_p$-ring for any flat $p$-adically complete $\Z_p$-algebra $C$, we must have $B \cong W^{(p)}(B/p)$. Comparing rank $1$ elements then gives a monoid isomorphism $(B/p)\cong \Lim_{x \mapsto x^p} C$, from which it follows that 
\[
 \bF_p\ten_{\bZ_p}\sA_{\inf}(C)\cong \Lim_{\Phi} (C/p)=C^{\flat}
\]
 whenever $C$ is perfectoid. Since tilting gives an equivalence of almost algebras by \cite[Theorem 5.2]{scholzePerfectoidSpaces}, this completes the proof.
\end{proof}

\subsection{Functoriality of $q$-de Rham cohomology}

 Since $(\Psi^{1/p^{\infty}}\widehat{\qDR}_p(A/\Z_p))^{\wedge_p}$ is represented by a cosimplicial  perfectoid $\Lambda_p$-ring over $\Z[q^{1/p^{\infty}}]^{\wedge_{(p,q-1)}}$ for any flat $\Lambda_p$-ring $A$ over $\Z_p$, it corresponds under Lemma \ref{cfperfectoid} to a cosimplicial perfectoid $(\Z[\zeta_{p^{\infty}}]^{\wedge_p})^a$-algebra, representing the following functor:

\begin{lemma}\label{perfectoidreplemma}
 For a perfectoid $(\Z[\zeta_{p^{\infty}}]^{\wedge_p})^a$-algebra $C$,   and a $\Lambda_p$-ring $A$ over $\Z_p$ with $X=\Spec A$, there is a canonical isomorphism
\[
 X_{\strat}^{q,p}(\sA_{\inf} (C)_*) \cong \im( \Lim_{\Psi^p} X(C_*) \to X(C_*)), 
\]
for the ring $C_*$ of almost elements.
\end{lemma}
\begin{proof}
By definition, $X_{\strat}^{q,p}(\sA_{\inf} (C)_*)$ is the image of 
\[
 \Hom_{\Lambda_p}(A, \sA_{\inf} (C)_*) \to \Hom_{\Lambda_p}(A, (\sA_{\inf} (C)_*)/(q-1)).
\]
Since right adjoints commute with limits and $\sA_{\inf}= \Lim_{\Psi^p}W^{(p)}$, we may rewrite the first term as $\Lim_{\Psi^p} \Hom_{\Lambda_p}(A, W^{(p)} (C_*))= \Lim_{\Psi^p} X(C_*)$.

Setting $B:=\Lim_{\Psi^p}W^{(p)} (C)_*$, observe that  because $[p^n]_{q^{1/p^n}}(q^{1/p^n}-1)=(q-1)$, we have  $\bigcap_n [p^n]_{q^{1/p^n}}B=(q-1)B$, any element on the left defining an almost element of $(q-1)B$, hence a genuine element since $B=B_*$ is flat.
 Then note that since the projection map $\theta \co B \to C_*$ has kernel $([p]_{q^{1/p}})$, the map $\theta \circ \Psi^{p^{n-1}}$ has kernel $([p]_{q^{1/p^n}})$, and so $B \to W^{(p)}(C)_*$ has kernel  $\bigcap_n [p^n]_{q^{1/p^n}}B$. Thus 
\[
 \Hom_{\Lambda_p}(A, (\Lim_{\Psi^p}W^{(p)} (C)_*)/(q-1))\into  \Hom_{\Lambda_p}(A,W^{(p)}(C)_*)= X(C_*).\qedhere
\]
\end{proof}
In fact, the tilting equivalence gives $\Lim_{\Psi^p} X(C_*) \cong X(C^{\flat}_*)$, so 
the only dependence of  $X_{\strat}^{q,p}(\sA_{\inf} (C)_*)$, and hence  $((\Psi^{1/p^{\infty}}\widehat{\qDR}_p(A/\Z_p))^{\wedge_p})^a$, on the Frobenius lift $\Psi^p$ is in determining the image of $ X(C^{\flat}_*) \to  X(C_*)$ as $C$ varies.

Although the map $ X(C^{\flat}_*) \to  X(C_*)$  is not surjective, it is almost so in a precise sense, which we now use to establish independence of $\Psi^p$, showing that, up to faithfully flat descent, $\widehat{\qDR}_p(A/\Z_p)^{\wedge_p}/[p]_{q^{1/p}}$ is the best possible perfectoid approximation to $A[\zeta_{p^{\infty}}]^{\wedge_p}$.

\begin{definition}
Given  a functor $X$ from  $(\Z[\zeta_{p^{\infty}}]^{\wedge_p})^a $-algebras to sets and a functor $\sA$ from perfectoid  $(\Z[\zeta_{p^{\infty}}]^{\wedge_p})^a$-algebras to abelian groups, we write
\[
 \oR\Gamma_{\Pfd}(X,\sA):=\oR\Hom_{[\Pfd((\Z_p[\zeta_{p^{\infty}}]^{\wedge_p})^a),\Set]}( X , \sA),
\]
where $\Pfd(S^a) $ denotes the category of perfectoid almost $S$-algebras, and $\oR\Hom_{[\C,\Set]}(-,-)$ is as in Definition \ref{RHomdef}.

When $X$ is representable by a $(\Z[\zeta_{p^{\infty}}]^{\wedge_p})^a $-algebra $C$, we simply denote $ \oR\Gamma_{\Pfd}(X,\sA)$ by $ \oR\Gamma_{\Pfd}(C,\sA)$ --- when  $C$ is perfectoid, this  will just be $\sA(C)$. 
\end{definition}
Thus $ \oR\Gamma_{\Pfd}(C,\sA)$ is the homotopy limit of the functor $\sA$ (regarded as taking values in cochain complexes) on the category of perfectoid  $(\Z[\zeta_{p^{\infty}}]^{\wedge_p})^a$-algebras equipped with a map from $C$. This is closely related to the pushforward from the pro-\'etale site of the generic fibre, whose d\'ecalage for $\sA=\sA_{\inf}$ is the complex $A\Omega$ of \cite[Definition 9.1]{BhattMorrowScholze}.

\begin{theorem}\label{mainthm4}
If $R$ is a  $p$-adically complete $\Lambda_p$-ring over $\Z_p$, and $A$ 
 a formal $R$-deformation  of a smooth ring over $(R/p)$, then the complex
\[
\oR\Gamma_{\Pfd}((A[\zeta_{p^{\infty}}]\ten_R \Psi^{1/p^{\infty}}R)^{\wedge_p},\sA_{\inf}) 
\]
of 
$(\Psi^{1/p^{\infty}}R[q])^{\wedge_{(p,q-1)}}$-modules 
is almost quasi-isomorphic to $(\Psi^{1/p^{\infty}}\widehat{\qDR}_p(A/R))^{\wedge_p} $ for any $\Lambda_p$-ring structure on $A$ 
 coming from a framing over $R$ as in Theorem \ref{mainthm3}.
 \end{theorem}
\begin{proof}
Since passage to almost modules is an exact functor, it  follows from the definition of $\qDR_p$  that  the cochain complex $((\Psi^{1/p^{\infty}}\widehat{\qDR}_p(A/R))^{\wedge_p} )^a $ is given by   $\oR\Hom_{[f\hat{\Lambda}_p(R\llb q-1 \rrb),\Set]}(X_{\strat}^{q,p}, ((\Psi^{1/p^{\infty}}\sO)^{\wedge_p})^a)$ in the notation of Definition \ref{RHomdef}, where $f\hat{\Lambda}_p(R\llb q-1 \rrb) $ denotes the category of flat $(p,q-1)$-adically complete $\Lambda_p$-algebras over $R\llb q-1 \rrb$. 

Now note that $C \mapsto ((\Psi^{1/p^{\infty}}C)^{\wedge_p} )_*$ is left adjoint to the inclusion functor $i \co   \Pfd\Lambda_p(R\llb q-1 \rrb) \to f\hat{\Lambda}_p(R\llb q-1 \rrb)$ from  the category of integral perfectoid $\Lambda_p$-rings over $\Psi^{1/p^{\infty}}(R\llb q-1 \rrb)^{\wedge_p}$. Thus $i^* \co \Ch([f\hat{\Lambda}_p(R\llb q-1 \rrb),\Ab]) \to  \Ch([\Pfd\Lambda_p(R\llb q-1 \rrb),\Ab])$ has exact right adjoint  $\sF \mapsto ( \sF \circ (\Psi^{1/p^{\infty}})^{\wedge_p})_*)$. 
We therefore have 
\[
 \oR\Hom_{[f\hat{\Lambda}_p(R\llb q-1 \rrb),\Set]}(X_{\strat}^{q,p}, ((\Psi^{1/p^{\infty}}\sO)^{\wedge_p})^a)\simeq \oR\Hom_{[\Pfd\Lambda_p(R\llb q-1 \rrb),\Set]}(i^*X_{\strat}^{q,p},\sO^a).
\]

It thus follows 
that  the cochain complex 
$((\Psi^{1/p^{\infty}}\widehat{\qDR}_p(A/R))^{\wedge_p} )^a $ 
is   the homotopy limit of the functor $(B,x,y) \mapsto B^a$ on the category of triples $(B,x,y)$ for integral perfectoid  $\Lambda_p$-rings $B$  over
 $\Z[q^{1/p^{\infty}}]^{\wedge_{(p,q-1)}}$ and 
\[
(x,y) \in X_{\strat}^{q,p}(B)\by_{Y_{\strat}^{q,p}(B)}Y(B),
\]
where $X=\Spec A$ and $Y=\Spec R$.

By Lemma \ref{cfperfectoid}, such $\Lambda_p$-rings $B$ are uniquely of the form $\sA_{\inf}(C_*)$ for $C \in \Pfd((\Z_p[\zeta_{p^{\infty}}]^{\wedge_p})^a)$, so this homotopy limit becomes
\[
((\Psi^{1/p^{\infty}}\widehat{\qDR}_p(A/R))^{\wedge_p})^a \simeq \oR\Gamma_{\Pfd}((X_{\strat}^{q,p} \by_{Y_{\strat}^{q,p}}Y) \circ (\sA_{\inf})_*, (\sA_{\inf}))^a.
\]

 Writing $X^{\infty}(C):= \im( \Lim_{\Psi^p} X(C_*) \to X(C_*))$,  Lemma \ref{perfectoidreplemma} then combines with  the description above to give
\begin{align*}
 (\widehat{\qDR}_p(A/R)^{\wedge_p})^a &\simeq \oR\Gamma_{\Pfd}(X^{\infty}\by_{Y^{\infty}}\Lim_{\Psi^p}Y, (\sA_{\inf}))^a,\\
&\simeq \oR\Gamma_{\Pfd}(X^{\infty}\by_Y\Lim_{\Psi^p}Y, (\sA_{\inf}))^a.
\end{align*}

We now introduce a Grothendieck topology on the category $[ \Pfd_{(\Z[\zeta_{p^{\infty}}]^{\wedge_p})^a},\Set]$  by taking covering morphisms to be those maps $C \to C'$ of perfectoid algebras 
which are almost faithfully flat  modulo $p$. Since $C^{\flat}=\Lim_{\Phi} (C/p)$, the functor $\sA_{\inf}$ satisfies descent with respect to these coverings, so 
the map 
\[
 \oR\Gamma_{\Pfd}((X^{\infty}\by_Y\Lim_{\Psi^p}Y)^{\sharp}, \sA_{\inf})^a\to
\oR\Gamma_{\Pfd}(X^{\infty}\by_Y\Lim_{\Psi^p}Y, \sA_{\inf})^a 
\]
is a quasi-isomorphism, where $(-)^{\sharp}$ denotes sheafification. 

In other words, the calculation of $(\widehat{\qDR}_p(A/R)^{\wedge_p})^a$ is not affected if  we  tweak the definition of $X^{\infty}$ by taking the image sheaf instead of the image presheaf.  
We then have
\[
 (X^{\infty})^{\sharp}(C)=\bigcup_{C\to C'} \im( X(C_*)\by_{X(C'_*)} \Lim_{\Psi^p} X(C_*')\to X(C_*)), 
\]
where $C \to C'$ runs over all covering  morphisms.  

Now, $\Lim_{\Psi^p} X$ is represented by the perfectoid algebra $(\Psi^{1/p^{\infty}}A)^{\wedge_p}$, which  is isomorphic to $A[x_1^{1/p^{\infty}}, \ldots,x_d^{1/p^{\infty}}]^{\wedge_p}$ as in the proof of  Lemma \ref{perfectioncoholemma}. This allows us to appeal to Andr\'e's results \cite[\S 2.5]{andreFacteur} as generalised in \cite[Theorem 2.3]{bhattDirectSummand}. 
For any  morphism $f \co A \to C$,  there exists a 
covering
morphism $C \to C_i$ 
such that $f(x_i)$ has  arbitrary $p$-power roots in $C_i$. Setting $C':= C_1\ten_C \ldots \ten_C C_d$, this means that the composite $A \xra{f} C \to C'$ extends to a map $(\Psi^{1/p^{\infty}}A)^{\wedge_p} \to C'$, so $f \in (X^{\infty})^{\sharp}(C)$. We have thus shown that $(X^{\infty})^{\sharp}=X$, giving the required equivalence
\[
 ((\Psi^{1/p^{\infty}}\widehat{\qDR}_p(A/R))^{\wedge_p} )^a  \simeq \oR\Gamma_{\Pfd}(X\by_Y\Lim_{\Psi^p}Y, (\sA_{\inf})_*)^a.
\]
Finally, compatibility of these equivalences with the $(\Psi^{1/p^{\infty}}R[q])^{\wedge_{(p,q-1)}}$-module structures is given by functoriality, multiplicativity and the identification $ (\Psi^{1/p^{\infty}}R[q])^{\wedge_{(p,q-1)}} \simeq (\Psi^{1/p^{\infty}}\widehat{\qDR}_p(R/R))^{\wedge_p}$. 
\end{proof}

\begin{remark}\label{qconnrmk2}
Corresponding to the cohomology theory $((\Psi^{1/p^{\infty}}\widehat{\qDR}_p(A/R))^{\wedge_p})^a$, it is natural to consider $q$-connections on finite projective modules $M$ over 
\begin{align*}
& \eta_{(q-1)}^0 ((\Psi^{1/p^{\infty}}(\Omega^*_{A/R}\llb q -1\rrb, (q-1)\nabla_q))^{\wedge_p,a}) \\
&= \{a \in (\Psi^{1/p^{\infty}}(A\llb q -1\rrb))^{\wedge_{(p,q-1)},a}~:~ \nabla_q a \in (\Psi^{1/p^{\infty}}(\Omega^1_A\llb q -1\rrb))^{\wedge_{(p,q-1)},a}\}\\
&= ((\sum_n [p^n]_{q^{1/p^n}} \Psi^{1/p^n}A[ q^{1/p^{\infty}}])^{\wedge_{(p,q-1)}})^a.
\end{align*}
It follows from the proof of Proposition \ref{qconnprop} that these are equivalent, for $X=\Spec A$, to finite projective almost $(\Psi^{1/p^{\infty}}\sO_{\hat{X}^q,\strat})^{\wedge_p}$-modules $\sN$ for which $\sN/(q-1)$ is the pullback of the almost $\H^0((\Psi^{1/p^{\infty}}\widehat{\qDR}_p(A/R))^{\wedge_p}/(q-1))$-module $\Gamma(\hat{X}^q_{\strat}, \sN/(q-1))=:M_0$. 

Up to almost isomorphism, these correspond via the proof of Theorem \ref{mainthm4}  to those finite  projective $\sA_{\inf}$-modules $N$ on the site  of integral perfectoid algebras $C$ over  $A[\zeta_{p^{\infty}}]^{\wedge_p}\ten_R \Psi^{1/p^{\infty}}R$ for which there exists a $W^{(p)}(A[\zeta_{p^{\infty}}]^{\wedge_p})$-module $M_0$ with $W^{(p)}(C)$-linear isomorphisms
\[
N(C)\ten_{\sA_{\inf}(C)}W^{(p)}(C) \cong M_0\ten_{W^{(p)}(A[\zeta_{p^{\infty}}]^{\wedge_p})}W^{(p)}(C),
\]
functorial in $C$.

This establishes a weakened form of \cite[Conjecture 7.5]{scholzeqdef} on co-ordinate independence of the category of $q$-connections, giving the statement for almost $(\sum_n [p^n]_{q^{1/p^n}} \Psi^{1/p^n}A[ q^{1/p^{\infty}}-1])^{\wedge_p})$-modules  rather than $A\llb q-1 \rrb$-modules.  
\end{remark}

The following gives a slight partial refinement of \cite[Theorem 1.17]{BhattMorrowScholze}: 
\begin{corollary}\label{finalcor}
If $R$ is a  $p$-adically complete $\Lambda_p$-ring over $\Z_p$, and $A$ 
 a formal $R$-deformation  of a smooth ring over $(R/p)$, then the $q$-de Rham cohomology complex $(q\mbox{-}\Omega^{\bt}_{A/R, \boxempty}\ten_{R[q]}(\Psi^{1/p^{\infty}}R)[q^{1/p^{\infty}}])^{\wedge_p}$ is, up to almost quasi-isomorphism, independent of a choice of co-ordinates $\boxempty$. As such, it is naturally an invariant of the commutative $p$-adically complete
 $(\Psi^{1/p^{\infty}}R)[\zeta_{p^{\infty}}]^{\wedge_p}$-algebra  $(A[\zeta_{p^{\infty}}]\ten_R \Psi^{1/p^{\infty}}R)^{\wedge_p}$.
\end{corollary}
\begin{proof}
 Since 
\[
 \Psi^{1/p^{\infty}}\qDR_p(A/R) = \Psi^{1/p^{\infty}}\qDR_p((A\ten_R \Psi^{1/p^{\infty}}R)/\Psi^{1/p^{\infty}}R),
\]
 Theorem \ref{mainthm3} combines with Lemma \ref{perfectioncoholemma}  to give 
\[
( q\mbox{-}\Omega^{\bt}_{A/R, \boxempty}\ten_{R[q]}(\Psi^{1/p^{\infty}}R)[q^{1/p^{\infty}}])^{\wedge_p} \simeq \oL\eta_{(q-1)}((\Psi^{1/p^{\infty}}\widehat{\qDR}_p(A/R))^{\wedge_p}), 
\]
and by Theorem \ref{mainthm4}, we know that this depends only on $(A[\zeta_{p^{\infty}}]\ten_R \Psi^{1/p^{\infty}}R)^{\wedge_p}$ up to almost quasi-isomorphism.
\end{proof}

\begin{remark}
The almost quasi-isomorphism in Corollary \ref{finalcor} should be a genuine quasi-isomorphism when we impose some conditions on the base ring $R$. By \cite[Lemma 8.11]{BhattMorrowScholze}, it would suffice to verify that $ \H^*((\Psi^{1/p^{\infty}} (\Omega^*_{A/R}\llbracket q-1 \rrbracket, (q-1)\nabla_q))^{\wedge_p})$ and its quotient by $(q-1)$ have no $(q^{1/p^{\infty}}-1)$-torsion, which should follow for $R$ smooth by an argument similar to \cite[Proposition 8.9]{BhattMorrowScholze}.
\end{remark}

\begin{remark}[Eliminating roots of $q$]\label{finalrmk}
The key feature of the comparison results in this section is that, up to faithfully flat descent, the functor $X_{\strat}^{q,p}$ does not depend on Adams operations when restricted to the category of  integral
 perfectoid  $\Lambda_p$-rings $B$ over $\Z[q]$, since the proof of Theorem \ref{mainthm4} gives $(X_{\strat}^{q,p})^{\sharp}(B) \cong X(B/[p]_{q^{1/p}})$.  We can extend the latter functor to more general $\Lambda_p$-rings over $\Z[q]$ by setting 
\[
 X^{q,p}(B):=X(B/(\Psi^p)^{-1}([p]_qB)),
\]
which does not  depend on any Adams operations on $X$.

When $\sO_X$ has a $\Lambda_p$-ring structure, there is then a natural map $\alpha \co  X_{\strat}^{q,p}\to X^{q,p}$  because $\Psi^p((q-1)B) \subset [p]_qB$. This induces a transformation
\[
 \alpha^* \co \oR\Hom_{[f\hat{\Lambda}_p(R\llb q-1 \rrb),\Set]}(X^{q,p},\sO) \to 
\widehat{\qDR}_p(A/R)^{\wedge_p}
\]
for $X=\Spec A$. But for integral perfectoid $\Lambda_p$-rings $B$, we know that $X^{q,p}(B)=  (X_{\strat}^{q,p})^{\sharp}(B)$, so by adjunction, as in the proof of Theorem \ref{mainthm4}, $\alpha^*$ becomes an almost quasi-isomorphism on applying a form of  completed stabilisation $\Psi^{1/p^{\infty}}(-)^{\wedge_p}$. 
 Thus $\H^*(X^{q,p},\sO)$ might be a candidate for the co-ordinate independent $q$-de Rham cohomology theory proposed in \cite{scholzeqdef}. It naturally carries an Adams operation $\Psi^p$, which would correspond to the operation $\phi_p$ of \cite[Conjecture 6.1]{scholzeqdef}. 

Any $a \in A$ defines an element of $\H^0(X^{q,p},\sO/(\Psi^p)^{-1}([p]_q\sO) )$ so $\Psi^p(a) \in\H^0(X^{q,p},\sO/[p]_q)$ and applying the connecting homomorphism associated to   
$[p]_q \co \sO \to \sO$
gives an element $\beta_{[p]_q} \Psi^p(a) \in\H^1(X^{q,p},\sO)$ whose image under $\H^1(\alpha^*)$ is 
\[
[p]_q^{-1}\Psi^p( (q-1)\nabla_qa)=(q-1)\Psi^p(\nabla_qa).
\]

Moreover, to $a \in A$ we may associate elements $a_n \in  \H^0(X^{q,p}, \sO/[p^{n}]_q)$ for $n \ge 1$, determined by the property that 
$a_n \equiv \Psi^{p^i}a^{p^{n-i}} \mod [p]_{q^{p^{i-1}}}$ for $1 \le i \le n$,
and these give rise to   elements $\beta_{[p^{n}]_q} a_n \in  \H^1(X^{q,p},\sO)$.
Explicitly, if we define  operations $\vareps_i$ on $\sO$ by  $\vareps_0=\id$ and  $\vareps_{i+1}(a):= (a^{p^{i+1}}- \Psi^p(a^{p^{i}}))/p^{i+1} $, then for a local lift $\tilde{a} \in \sO$ of $a \in \sO/(\Psi^p)^{-1}([p]_q)$, we have
\[
a_n = \sum_{i=0}^{n-1} [p^i]_{q^{p^{n-i}}}\Psi^{p^{n-i}}( \vareps_i\tilde{a}) +  [p^n]_q\sO, 
\]
so
\begin{align*}
 \H^1(\alpha^*)(\beta_{[p^{n}]_q} a_n)
&=   (q-1) \sum_{i=0}^{n-1}  \Psi^{p^{n-i}}( \nabla_q \vareps_i \tilde{a}).
\end{align*}
 
In particular, 
for $A=R[x]$ these include all the elements $ (q-1) [m]_{q^{p^{n}}}x^{p^{n} m -1}dx $, since $\vareps_i(x^m)=0$ for all $i>0$, $x^m$ having rank $1$.
This  suggests  that in general the image of  $\H^1(\alpha^*)$  might be $(q-1)\H^1\widehat{\qDR}_p(A/R)^{\wedge_p}$, tying in well with $(q-1)$-adic d\'ecalage. Explicit descriptions for much of the functoriality from Corollary \ref{finalcor} can also be inferred from this analysis, since it implies that the transformations $  \sum_{i=0}^{n-1}  \Psi^{p^{n-i}} \circ \nabla_q \circ \vareps_i\co A \to  \H^1(q\mbox{-}\Omega^{\bt}_{A/R, \boxempty})$ are all natural  in $A$. 
\end{remark}

\bibliographystyle{alphanum}
\bibliography{references.bib}

\def\cprime{$'$}
\begin{thebibliography}{BMS}

\bibitem[And]{andreFacteur}
Yves Andr\'{e}.
\newblock La conjecture du facteur direct.
\newblock {\em Publ. Math. Inst. Hautes \'{E}tudes Sci.}, 127:71--93, 2018.
\newblock arXiv: 1609.00345 [math.AG].

\bibitem[Bd]{BhattDeJong}
B.~{Bhatt} and A.~J. {de Jong}.
\newblock {Crystalline cohomology and de Rham cohomology}.
\newblock arXiv: 1110.5001 [math.AG], 2011.

\bibitem[Bha]{bhattDirectSummand}
Bhargav Bhatt.
\newblock On the direct summand conjecture and its derived variant.
\newblock {\em Invent. Math.}, 212(2):297--317, 2018.
\newblock arXiv:1608.08882[math.AG].

\bibitem[BK]{bousfieldkan}
A.~K. Bousfield and D.~M. Kan.
\newblock {\em Homotopy limits, completions and localizations}.
\newblock Lecture Notes in Mathematics, Vol. 304. Springer-Verlag, Berlin,
  1972.

\bibitem[BMS]{BhattMorrowScholze}
B.~Bhatt, M.~Morrow, and P.~Scholze.
\newblock Integral $p$-adic {H}odge theory.
\newblock arXiv:1602.03148 [math.NT], 2016.

\bibitem[Bor]{borgerBasicGeomI}
James Borger.
\newblock The basic geometry of {W}itt vectors, {I}: {T}he affine case.
\newblock {\em Algebra Number Theory}, 5(2):231--285, 2011.

\bibitem[GR]{gabberalmost}
Ofer Gabber and Lorenzo Ramero.
\newblock {\em Almost ring theory}, volume 1800 of {\em Lecture Notes in
  Mathematics}.
\newblock Springer-Verlag, Berlin, 2003.
\newblock arXiv:math/0201175v3 [math.AG].

\bibitem[Gro]{Gr}
A.~Grothendieck.
\newblock Crystals and the de {R}ham cohomology of schemes.
\newblock In {\em Dix Expos{\'e}s sur la Cohomologie des Sch{\'e}mas}, pages
  306--358. North-Holland, Amsterdam, 1968.

\bibitem[Hir]{Hirschhorn}
Philip~S. Hirschhorn.
\newblock {\em Model categories and their localizations}, volume~99 of {\em
  Mathematical Surveys and Monographs}.
\newblock American Mathematical Society, Providence, RI, 2003.

\bibitem[Joy]{joyaldeltaWitt}
Andr{\'e} Joyal.
\newblock {$\delta$}-anneaux et vecteurs de {W}itt.
\newblock {\em C. R. Math. Rep. Acad. Sci. Canada}, 7(3):177--182, 1985.

\bibitem[Mac]{mac}
Saunders MacLane.
\newblock {\em Categories for the working mathematician}.
\newblock Springer-Verlag, New York, 1971.
\newblock Graduate Texts in Mathematics, Vol. 5.

\bibitem[Sch1]{scholzePerfectoidSpaces}
Peter Scholze.
\newblock Perfectoid spaces.
\newblock {\em Publ. Math. Inst. Hautes \'Etudes Sci.}, 116:245--313, 2012.

\bibitem[Sch2]{scholzeqdef}
Peter Scholze.
\newblock Canonical {$q$}-deformations in arithmetic geometry.
\newblock {\em Ann. Fac. Sci. Toulouse Math. (6)}, 26(5):1163--1192, 2017.
\newblock arXiv:1606.01796v1[math.AG].

\bibitem[Sim]{simpsonHtpy}
Carlos Simpson.
\newblock Homotopy over the complex numbers and generalized de {R}ham
  cohomology.
\newblock In {\em Moduli of vector bundles ({S}anda, 1994; {K}yoto, 1994)},
  volume 179 of {\em Lecture Notes in Pure and Appl. Math.}, pages 229--263.
  Dekker, New York, 1996.

\bibitem[Wei]{W}
Charles~A. Weibel.
\newblock {\em An introduction to homological algebra}.
\newblock Cambridge University Press, Cambridge, 1994.

\bibitem[Wil]{wilkerson}
Clarence Wilkerson.
\newblock Lambda-rings, binomial domains, and vector bundles over {${\bf
  C}P(\infty )$}.
\newblock {\em Comm. Algebra}, 10(3):311--328, 1982.

\end{thebibliography}

\end{document}